\documentclass[a4paper,11pt]{article}
\usepackage[left=3cm,right=3cm,top=3cm,bottom=3cm]{geometry}

\usepackage[dvips]{epsfig}
\usepackage{epstopdf}
\usepackage{amsmath,amssymb,amsthm}
\usepackage{tikz}
\usepackage{enumerate}
\usetikzlibrary{decorations.markings,calc}
\usepackage{youngtab}
\usepackage[all]{xy}
\usepackage{array}

\newtheorem{theorem}{Theorem}
\numberwithin{theorem}{section}
\newtheorem{corollary}[theorem]{Corollary}
\newtheorem{proposition}[theorem]{Proposition}
\newtheorem{lemma}[theorem]{Lemma}
\newtheorem{varexample}[theorem]{Example}
\newenvironment{example}{\begin{varexample}\em}{\em\end{varexample}}
\newtheorem{definition}[theorem]{Definition}

\theoremstyle{definition}

\tikzset{blob/.style={circle, inner sep=2pt, draw, thick, fill}}
\tikzset{arrow/.style={decoration={
    markings,
    mark=at position #1 with \arrow{>}},
    postaction=decorate}
}
\tikzset{reverse arrow/.style={decoration={
    markings,
    mark=at position #1 with \arrow{<}},
    postaction=decorate}
}

\newcommand{\ra}{\mathop{\rightarrow}}

\newcommand{\inprod}[1]{\left\langle{#1}\right\rangle}

\newcommand\pdag{{\vphantom{\dagger}}}

\newcommand{\Sn}{S_n}
\newcommand{\Sni}{S_{n+1}}
\newcommand{\Cz}{\mathbb{C}[\![z]\!]}
\newcommand{\C}{\ensuremath{\mathbb{C}}}

\newcommand{\sln}{\mathrm{sl}_n}
\newcommand{\Usln}{U(\mathrm{sl}_n)}
\newcommand{\Uqsln}{U_q(\mathrm{sl}_n)}
\newcommand{\nCFSN}{\FSN^n}
\newcommand{\Czz}{\mathbb{C}[\![ z_1, \dots, z_n ]\!]}

\newcommand\id{\ensuremath{\mathrm{id}}}
\newcommand\Hom{\mathrm{Hom}}
\newcommand\Kar{\mathrm{Kar}}

\newcommand\Ob{\ensuremath{\mathrm{Ob}}}

\newcommand{\opname}[1]{\operatorname{#1}}
\newcommand{\cat}[1]{\ensuremath{\mathbf{#1}}}
\newcommand\inalign[1]{\begin{aligned}#1\end{aligned}}

\renewcommand{\_}[0]{\nobreakdash--\hspace{0pt}}

\newcommand{\V}{\cat{Vect}}

\newcommand{\iiV}{\cat{2Vect}}
\newcommand{\FSN}{{\cat{S}}}
\newcommand{\Gpd}{\cat{Gpd}}
\newcommand\Rep{\mathrm{Rep}}

\newcommand{\Span}{\opname{\bf Span}}

\newcommand{\SiiG}{{\bf\Span(\Gpd)}}
\newcommand*\sxto[1]{\xrightarrow{\smash{#1}}}
\newcommand*\sxlto[1]{\xleftarrow{\smash{#1}}}

\newcommand{\FV}{\Lambda}
\newcommand{\HV}[1]{[#1,\V]}
\newcommand{\A}{\mathbf{a}}

\newcommand{\injAsA}{i_{A^{\dagger} \circ A}}

\newcolumntype{E}{@{}c@{}}

\begin{document}

\setcounter{MaxMatrixCols}{12}

\title{The Categorified Heisenberg Algebra I:
\\
A Combinatorial Representation}
\author{Jeffrey C. Morton\footnote{This work partially financed by Portuguese funds via the Funda\c{c}\~{a}o para a Ci\^encia e a Tecnologia, through project number PTDC/MAT/101503/2008, New Geometry and Topology.}
\\
Department of Mathematics, University of Hamburg
\\
\texttt{jeffrey.c.morton@theoreticalatlas.net}
\\
\\
Jamie Vicary
\\
Centre of Quantum Technologies, University of Singapore
\\
and Department of Computer Science, University of Oxford
\\
\texttt{jamie.vicary@cs.ox.ac.uk}}

\maketitle

\begin{abstract}We give an introductory account of Khovanov's categorification of the Heisenberg algebra, and construct a combinatorial model for it in a 2-category of spans of groupoids. We also treat a categorification of $U(\mathrm{sl}_n)$ in a similar way. These give rise to standard representations on vector spaces through a linearization process.
\end{abstract}

\section{Introduction}

\subsection{Overview}

Our objective in this paper is to describe a combinatorial model of Khovanov's categorification of the Heisenberg algebra~\cite{khovanov}, using a natural construction based on the groupoidification programme of Baez and Dolan~\cite{finfeyn, hdavii}.  This gives rise to a narrative which serves to `explain' the structure of the categorified algebra in terms of the combinatorics of finite sets. It is also fundamental, giving rise to known linear representations via a canonical process of linearization. This account is one part of a theory of representations of the categorified Heisenberg algebra in terms of free symmetric algebraic structures, described in a forthcoming companion article~\cite{mortonvicary}.

The model of the categorified Heisenberg algebra which we describe is based on a groupoidification of the quantum harmonic oscillator, a simple quantum-mechanical system.  This system has an infinite-dimensional Hilbert space of states called \emph{Fock space}, which carries an action of a \emph{Heisenberg algebra}.  The simplest of these, describing a quantum harmonic oscillator with a single degree of freedom, is the free complex algebra on generators $\A^\dag$ and $\A$, called the \emph{creation} and \emph{annihilation} operators respectively, modulo the commutation relation
\begin{equation}
\label{eq:ccr}
\A \A ^\dagger - \A ^\dagger \A = 1.
\end{equation}
The action on Fock space involves unbounded operators, and equation~\eqref{eq:ccr} is only required to hold on a dense domain.

The Fock space for a single oscillator has a standard basis of states called the Fock basis, which are labeled by non-negative integer \textit{energy levels}.  Elements of this basis represent states in which the system contains an integer number of energy quanta, interpreted in a quantum field theory setting as `particles'. The creation and annihilation operators act on this basis to increase or decrease the number of particles by one.

This description is already understood to be a `shadow', or \textit{decategorification}, of a richer perspective in which the actual sets of particles are themselves the relevant mathematical objects~\cite{finfeyn, hdavii, mstuff}. Thus, for instance, rather than merely describing a change in particle number, one can instead consider the maps between sets of particles which \textit{witness} that change. This makes room for more structure, so that one can consider the action of the symmetric group which corresponds to physically permuting the particles, an act which has no nontrivial mathematical representation in the case of ordinary Fock space.  The usual Fock space picture is recovered in the decategorification (or `degroupoidification') of this structure.

The categorified Heisenberg algebra captures the interesting mathematical statements that can be made in this richer setting. The new contribution here is the interpretation of these statements in terms of the combinatorics of finite sets.

These structures have an abstract mathematical description which makes them more general than any one model. In particular, there are applications in computer science~\cite{differentiallambda, axiomaticsfiore, generalisedspecies} for which a Fock space--like structure represents an unlimited number of copies of some logical resource; the constructions of this paper are relevant to categorifications of this scenario. We will explore this more general perspective in a more technical article~\cite{mortonvicary}, where we show how representations of the categorified Heisenberg algebra arise generically on free symmetric pseudomonoid structures internal to monoidal 2\-categories with sufficiently good properties.

\subsection{Khovanov's categorification}
\label{sec:khovanov}

More technically, to \textit{categorify} a complex algebra means to find a monoidal category \cat{C} with coproducts, such that the algebra can be recovered as the complexification of the monoid of isomorphism classes of objects of \cat{C}, with the vector space structure on the algebra arising from the coproduct structure in \cat{C}. This can be seen as the composition of two mathematical processes:
\begin{equation}
  [\text{Monoidal Categories}] \sxto{K_0} [\text{Rings}] \sxto{\mathbb{C} \otimes -} [\text{Algebras}]
\end{equation}
Khovanov has given a categorification of the Heisenberg algebra in this sense~\cite{khovanov}, part of a broader programme which includes the  categorification of quantum groups~\cite{KL, webster-diag}.  In general, this programme involves the construction of monoidal categories whose morphisms are classes of diagrams, sometimes with
decorations of various kinds, modulo certain
topological identifications. 

Khovanov's categorification of the Heisenberg algebra takes the form of a monoidal category \cat{H'} with a zero object and biproducts, with generating objects $Q_+$ and $Q_-$ depicted as upwards- and downwards-pointing strands respectively:
\begin{equation*}
\begin{aligned}
\begin{tikzpicture}
\draw [arrow=0.5] (0,0) node [below] {$Q_+$} to (0,1);
\end{tikzpicture}
\end{aligned}
\hspace{100pt}
\begin{aligned}
\begin{tikzpicture}
\draw [reverse arrow=0.5] (0,0) node [below] {$Q_-$} to (0,1);
\end{tikzpicture}
\end{aligned}
\end{equation*}
Tensor product is represented by horizontal juxtaposition. The generating morphisms have the following graphical representations:
\newcommand{\centerdia}[1]{\makebox[2cm]{\ensuremath{#1}}}
\begin{align*}
\centerdia{\begin{aligned}
\begin{tikzpicture}
\draw [white] (0,0) to (0,-1.0);
\draw (0,0) to (0,-0.2) to [out=down, in=down, looseness=2] (1,-0.2) to (1,0);
\draw [<-] (0,-0.2) to (0,0);
\draw [-<] (1,0) to (1,-0.2);
\end{tikzpicture}
\end{aligned}}
&&
\centerdia{\begin{aligned}
\begin{tikzpicture}
\draw [white] (0,-0.2) to (0,0.8);
\draw (0,-0.2) to (0,0) to [out=up, in=up, looseness=2] (1,0) to (1,-0.2);
\draw [-<] (0,-0.2) to (0,0);
\draw [<-] (1,0) to (1,-0.2);
\end{tikzpicture}
\end{aligned}}
&&
\centerdia{\begin{aligned}
\begin{tikzpicture}
\draw [white] (0,0) to (0,-1.0);
\draw (0,0) to (0,-0.2) to [out=down, in=down, looseness=2] (1,-0.2) to (1,0);
\draw [>-] (0,-0.2) to (0,0);
\draw [->] (1,0) to (1,-0.2);
\end{tikzpicture}
\end{aligned}}
&&
\centerdia{\begin{aligned}
\begin{tikzpicture}
\draw [white] (0,-0.2) to (0,0.8);
\draw (0,-0.2) to (0,0) to [out=up, in=up, looseness=2] (1,0) to (1,-0.2);
\draw [->] (0,-0.2) to (0,0);
\draw [>-] (1,0) to (1,-0.2);
\end{tikzpicture}
\end{aligned}}
\end{align*}

\vspace{-10pt}
\begin{align*}
\centerdia{\begin{aligned}
\begin{tikzpicture}
\draw [arrow=0.1, arrow=0.9] (0,0) to [out=up, in=down] (1,1.5);
\draw [arrow=0.1, arrow=0.9] (1,0) to [out=up, in=down] (0,1.5);
\end{tikzpicture}
\end{aligned}}
&&
\centerdia{\begin{aligned}
\begin{tikzpicture}
\draw [reverse arrow=0.1, reverse arrow=0.9] (0,0) to [out=up, in=down] (1,1.5);
\draw [reverse arrow=0.1, reverse arrow=0.9] (1,0) to [out=up, in=down] (0,1.5);
\end{tikzpicture}
\end{aligned}}
&&
\centerdia{\begin{aligned}
\begin{tikzpicture}
\draw [arrow=0.1, arrow=0.9] (0,0) to [out=up, in=down] (1,1.5);
\draw [reverse arrow=0.1, reverse arrow=0.9] (1,0) to [out=up, in=down] (0,1.5);
\end{tikzpicture}
\end{aligned}}
&&
\centerdia{\begin{aligned}
\begin{tikzpicture}
\draw [reverse arrow=0.1, reverse arrow=0.9] (0,0) to [out=up, in=down] (1,1.5);
\draw [arrow=0.1, arrow=0.9] (1,0) to [out=up, in=down] (0,1.5);
\end{tikzpicture}
\end{aligned}}
\end{align*}
The following equations are then imposed between composites of these generating morphisms:

\vspace{10pt}
\def\tempscale{0.8}
\def\tempgap{3pt}
\noindent\begin{tabular*}{\textwidth}{@{\extracolsep{\fill}}E*{3}{>{$}c<{$}}E}
&
\begin{aligned}
\begin{tikzpicture}[scale=\tempscale]
\draw (0,-0.2) to [out=down, in=down, looseness=2] (1,-0.2);
\draw (0,-2.8) to [out=up, in=up, looseness=2] (1,-2.8);
\draw [->] (0,0) to (0,-0.2);
\draw [-<] (1,0) to (1,-0.2);
\draw [-<] (0,-3) to (0,-2.8);
\draw [->] (1,-3) to (1,-2.8);
\end{tikzpicture}
\end{aligned}
\hspace{\tempgap}+\hspace{\tempgap}
\begin{aligned}
\begin{tikzpicture}[scale=\tempscale]
\draw [white] (0,0.2) to (1,0.2);
\draw [reverse arrow=0.05, reverse arrow=0.5, reverse arrow=0.95] (0,0)
    to [out=up, in=down] (1,1.5) to [out=up, in=down] (0,3);
\draw [arrow=0.05, arrow=0.5, arrow=0.95] (1,0)
    to [out=up, in=down] (0,1.5) to [out=up, in=down] (1,3);
\end{tikzpicture}
\end{aligned}
\hspace{\tempgap}=\hspace{\tempgap}
\begin{aligned}
\begin{tikzpicture}[scale=\tempscale]
\draw [white] (0,0.2) to (1,0.2);
\draw [reverse arrow=0.5] (0,0)to (0,3);
\draw [arrow=0.5] (1,0) to (1,3);
\end{tikzpicture}
\end{aligned}
&
\begin{aligned}
\begin{tikzpicture}[scale=\tempscale]
\draw [arrow=0.05, arrow=0.5, arrow=0.95] (0,0)
    to (0,1)
    to [out=up, in=down] (-1,2)
    to [out=up, in=up, looseness=1.5] (-2,2)
    to (-2,1)
    to [out=down, in=down, looseness=1.5] (-1,1)
    to [out=up, in=down] (0,2)
    to (0,3);
\end{tikzpicture}
\end{aligned}
\hspace{\tempgap}=\hspace{\tempgap}
0
&
\begin{aligned}
\begin{tikzpicture}[scale=\tempscale]
\draw [reverse arrow=0, reverse arrow=0.5] (0,0)
    to [out=up, in=up, looseness=2] (1,0)
    to [out=down, in=down, looseness=2] (0,0);
\end{tikzpicture}
\end{aligned}
\hspace{\tempgap}=\hspace{\tempgap}
\id
&
\end{tabular*}

\vspace{10pt}
\noindent\begin{tabular*}{\textwidth}{@{\extracolsep{\fill}}E*{4}{>{$}c<{$}}E}
&
\begin{aligned}
\begin{tikzpicture}[scale=\tempscale]
\draw [white] (0,0.2) to (1,0.2);
\draw [reverse arrow=0.05, reverse arrow=0.5, reverse arrow=0.95] (0,0)
    to [out=up, in=down] (1,1.5) to [out=up, in=down] (0,3);
\draw [reverse arrow=0.05, reverse arrow=0.5, reverse arrow=0.95] (1,0)
    to [out=up, in=down] (0,1.5) to [out=up, in=down] (1,3);
\end{tikzpicture}
\end{aligned}
\hspace{\tempgap}=\hspace{\tempgap}
\begin{aligned}
\begin{tikzpicture}[scale=\tempscale]
\draw [white] (0,0.2) to (1,0.2);
\draw [reverse arrow=0.5] (0,0)to (0,3);
\draw [reverse arrow=0.5] (1,0) to (1,3);
\end{tikzpicture}
\end{aligned}
&
\begin{aligned}
\begin{tikzpicture}[scale=\tempscale]
\draw [white] (0,0.2) to (1,0.2);
\draw [arrow=0.05, arrow=0.5, arrow=0.95] (0,0)
    to [out=up, in=down] (1,1.5) to [out=up, in=down] (0,3);
\draw [arrow=0.05, arrow=0.5, arrow=0.95] (1,0)
    to [out=up, in=down] (0,1.5) to [out=up, in=down] (1,3);
\end{tikzpicture}
\end{aligned}
\hspace{\tempgap}=\hspace{\tempgap}
\begin{aligned}
\begin{tikzpicture}[scale=\tempscale]
\draw [white] (0,0.2) to (1,0.2);
\draw [arrow=0.5] (0,0)to (0,3);
\draw [arrow=0.5] (1,0) to (1,3);
\end{tikzpicture}
\end{aligned}
&
\begin{aligned}
\begin{tikzpicture}[scale=\tempscale]
\draw [white] (0,0.2) to (1,0.2);
\draw [arrow=0.05, arrow=0.5, arrow=0.95] (0,0)
    to [out=up, in=down] (1,1.5) to [out=up, in=down] (0,3);
\draw [reverse arrow=0.05, reverse arrow=0.5, reverse arrow=0.95] (1,0)
    to [out=up, in=down] (0,1.5) to [out=up, in=down] (1,3);
\end{tikzpicture}
\end{aligned}
\hspace{\tempgap}=\hspace{\tempgap}
\begin{aligned}
\begin{tikzpicture}[scale=\tempscale]
\draw [white] (0,0.2) to (1,0.2);
\draw [arrow=0.5] (0,0)to (0,3);
\draw [reverse arrow=0.5] (1,0) to (1,3);
\end{tikzpicture}
\end{aligned}
&
\begin{aligned}
\begin{tikzpicture}[xscale=\tempscale, yscale=1.5*\tempscale]
\draw [arrow=0.08, arrow=0.92] (0,0)
    to [out=up, in=down] (1,1)
    to [out=up, in=down] (0,2);
\draw [arrow=0.08, arrow=0.92] (-1,0) to [out=up, in=down] (1,2);
\draw [arrow=0.08, arrow=0.92] (1,0) to [out=up, in=down] (-1,2);
\end{tikzpicture}
\end{aligned}
\hspace{\tempgap}=\hspace{\tempgap}
\begin{aligned}
\begin{tikzpicture}[xscale=\tempscale, yscale=1.5*\tempscale]
\draw [arrow=0.08, arrow=0.92] (0,0)
    to [out=up, in=down] (-1,1)
    to [out=up, in=down] (0,2);
\draw [arrow=0.08, arrow=0.92] (-1,0) to [out=up, in=down] (1,2);
\draw [arrow=0.08, arrow=0.92] (1,0) to [out=up, in=down] (-1,2);
\end{tikzpicture}
\end{aligned}
&
\end{tabular*}

\vspace{10pt}
\noindent\begin{tabular*}{\textwidth}{@{\extracolsep{\fill}}E*{4}{>{$}c<{$}}E}
&
\begin{aligned}
\begin{tikzpicture}[scale=\tempscale]
\draw [arrow=0.05, arrow=0.5, arrow=0.95] (0,0) to (0,1.5) to [out=up, in=up, looseness=2] (1,1.5) to [out=down, in=down, looseness=2] (2,1.5) to (2,3);
\end{tikzpicture}
\end{aligned}
\hspace{\tempgap}=\hspace{\tempgap}
\begin{aligned}
\begin{tikzpicture}[scale=\tempscale]
\draw [arrow=0.5] (0,0) to (0,3);
\end{tikzpicture}
\end{aligned}
&
\begin{aligned}
\begin{tikzpicture}[scale=\tempscale]
\draw [reverse arrow=0.05, reverse arrow=0.5, reverse arrow=0.95] (0,0) to (0,1.5) to [out=up, in=up, looseness=2] (-1,1.5) to [out=down, in=down, looseness=2] (-2,1.5) to (-2,3);
\end{tikzpicture}
\end{aligned}
\hspace{\tempgap}=\hspace{\tempgap}
\begin{aligned}
\begin{tikzpicture}[scale=\tempscale]
\draw [reverse arrow=0.5] (0,0) to (0,3);
\end{tikzpicture}
\end{aligned}
&
\begin{aligned}
\begin{tikzpicture}[scale=\tempscale]
\draw [reverse arrow=0.05, reverse arrow=0.5, reverse arrow=0.95] (0,0) to (0,1.5) to [out=up, in=up, looseness=2] (1,1.5) to [out=down, in=down, looseness=2] (2,1.5) to (2,3);
\end{tikzpicture}
\end{aligned}
\hspace{\tempgap}=\hspace{\tempgap}
\begin{aligned}
\begin{tikzpicture}[scale=\tempscale]
\draw [reverse arrow=0.5] (0,0) to (0,3);
\end{tikzpicture}
\end{aligned}
&
\begin{aligned}
\begin{tikzpicture}[scale=\tempscale]
\draw [arrow=0.05, arrow=0.5, arrow=0.95] (0,0) to (0,1.5) to [out=up, in=up, looseness=2] (-1,1.5) to [out=down, in=down, looseness=2] (-2,1.5) to (-2,3);
\end{tikzpicture}
\end{aligned}
\hspace{\tempgap}=\hspace{\tempgap}
\begin{aligned}
\begin{tikzpicture}[scale=\tempscale]
\draw [arrow=0.5] (0,0) to (0,3);
\end{tikzpicture}
\end{aligned}
&
\end{tabular*}
\vspace{10pt}

\noindent
These crossings cannot be interpreted as braidings as they are not invertible, as emphasized by the top-left diagram. These equations imply an isomorphism
\begin{equation}
Q_- \otimes Q_+ \simeq (Q_+ \otimes Q_-) \oplus I
\end{equation}
where $I$ is the monoidal unit object, giving a categorification of the Heisenberg algebra relation~\eqref{eq:ccr}.

Khovanov goes on to show that this has a representation on a monoidal category whose objects are bimodules describing restriction and induction of representations of symmetric groups.  As we will see, this amounts to a representation on \emph{2\_vector spaces} in the sense of Kapranov and Voevodsky, giving a 2\-functor
\begin{equation}
\label{eq:khovanovrepresentation}
\cat{\Omega_{H'}} \sxto K \cat{2Vect}
\end{equation}
where $\cat{\Omega_{H'}}$ is a 2\-category with one object, whose morphisms and 2-morphisms come from the objects and morphisms of \cat{H'}. This takes the single object of $\cat{\Omega_{H'}}$ to a 2\_vector space $\coprod_n \mathrm{Rep}(S_n)$, the coproduct of the categories of representations of the symmetric groups, which plays the role of a categorified Fock space.

We will restrict our attention to this monoidal category $\cat{H'}$ of diagrams in Section \ref{sec:groupoidification}, when we give them a combinatorial interpretation. However, in fact, the categorification of the full Heisenberg algebra with multiple generators $a_n$ requires somewhat more.  The generators $a_n$ are represented in the categorification by permutation-invariant subobjects of products  of the basic objects $Q_+$ and $Q_-$.  In the diagrammatic categorification, this is done formally by completing $\cat{H'}$ to a category called $\cat{H}$, in which all the required subobjects of $Q_+^n$ exist.  In Section \ref{sec:linear}, we will turn to this in more detail, and in particular \ref{sec:symmetrizers} discusses how the particular symmetrizer subobjects which appear in this completion occur naturally in $\cat{2Vect}$ as sub-functors, and describes some of them in detail.

\subsection{Groupoidification}

The concrete model of $\cat{H'}$ we find arises from a seemingly different approach to categorifying the Heisenberg algebra, based on the \textit{groupoidification} program of Baez and Dolan~\cite{hdavii}. A \textit{groupoid} is a category in which all morphisms are invertible. The central objects of study in the groupoidification programme are \textit{spans of groupoids}, diagrams of the form
\begin{equation}
\begin{aligned}
\begin{tikzpicture}[xscale=2, yscale=1.5]
\node (T) at (0,0) {$\cat X$};
\node (L) at (-1,-1) {$\cat B$};
\node (R) at (1,-1) {$\cat A$};
\draw [->] (T) to node [auto, swap] {$G$} (L);
\draw [->] (T) to node [auto] {$F$} (R);
\end{tikzpicture}
\end{aligned}
\end{equation}
where \cat A, \cat B and \cat X are groupoids, and $F$ and $G$ are functors. There has been extensive work within category theory on constructions involving spans~\cite{spanconstruction}.  The most immediately important fact is that spans of groupoids can be organized into a 2\-category $\SiiG$. The main construction of this paper is a 2-functor
\begin{equation}
\label{eq:combinatorialrepresentation}
\cat{\Omega_{H'}} \sxto C \cat{Span(Gpd)},
\end{equation}
yielding a representation of Khovanov's categorification of the Heisenberg algebra in terms of spans of groupoids.

An important interpretation of groupoidification comes from physics: groupoids represent \emph{physical symmetries}, and spans represent \textit{spaces of histories}. The idea is that configuration spaces of physical systems can be represented as groupoids in a way that usefully encodes their symmetries, and that spans of groupoids encode ways in which these states and their symmetries can be transformed via physical processes.  In our example of the harmonic oscillator, these configurations are the non-negative integer--valued energy eigenstates.  Then a span represents a space of \textit{histories}, with its source and target maps picking out the starting and ending configurations.  Furthermore, these are not just set-maps of the objects (histories and configurations), but functors, which also describe how symmetries of histories act on starting and ending configurations.

The single object of \cat{H'} is mapped by $C$ to the groupoid \FSN{} of finite sets and bijections. Spans involving this groupoid
have an elegant interpretation in term of the combinatorics of finite sets, thanks to work of Joyal~\cite{joyal1986} and Baez and Dolan~\cite{finfeyn}. The resulting combinatorial interpretation of our representation gives us a new perspective on the categorified Heisenberg algebra.

A construction called \emph{2-linearization} \cite{gpd2vect} converts a span of groupoids to a 2--linear map between 2--vector spaces. This resulting 2--linear map can be thought of as being `accounted for' by the underlying span of groupoids, giving rise to a useful combinatorial perspective on the mathematics.
 The 2\-linearization process gives a 2\-functor of the following type:
\begin{equation}
\cat{Span(Gpd)} \sxto\Lambda \cat{2Vect}
\end{equation}
Composing this functor with our representation~\eqref{eq:combinatorialrepresentation} gives us a linear representation of the categorified Heisenberg algebra:
\begin{equation}
\cat{\Omega_{H'}} \sxto C \cat{Span(Gpd)} \sxto \Lambda \cat{2Vect}
\end{equation}
We will see that this reproduces Khovanov's representation~\eqref{eq:khovanovrepresentation}, and hence can be seen as giving a concrete combinatorial `explanation' of his construction.

In Section~\ref{sec:sln} we extend these ideas to produce a groupoidification of the universal enveloping algebras $U(\mathrm{sl}_n)$ of the Lie algebras $\mathrm{sl}_n$. Since these algebras are closely related to the Heisenberg algebra it is perhaps not surprising that such a treatment can be given. The treatments of  both families of algebras is unified in the accompanying article~\cite{mortonvicary}, in which the representations of the categorified Heisenberg algebra and categorified $U(\mathrm{sl}_n)$ are both seen as arising from free symmetric monoidal groupoids: for the Heisenberg algebra, on the trivial groupoid with one morphism; and for $U(\mathrm{sl}_n)$, on the discrete groupoid with $n$ morphisms. The construction relating Heisenberg algebras and $U(\mathrm{sl}_n)$ is related to Kac-Moody algebras and their categorifications~\cite{rouquier-km}, so groupoidification may also provide a useful perspective in this case.

In the $q$-deformed case,  categorifications in terms of 2\_vector spaces have been described as a part of the Khovanov-Lauda programme~\cite{KL}. The groupoidification formalism used here, based on the groupoid of finite sets and bijections,  cannot be directly applied in this case.  Baez, Hoffnung and Walker~\cite{hdavii} have suggested that such groupoidifications should  not be in terms of bijections of sets, but rather linear bijections of vector spaces over the finite field $\mathbb{F}_q$ with $q$ elements, for $q$ a prime power.  These would used as groupoidifications of the Hecke algebras which appear in place of symmetric group algebras in the categorifications of quantum groups. If successful, this would yield combinatorial models for these categorified $q$-deformed algebras, yielding further insight into their structure.

\subsection*{Acknowledgements}

The authors are grateful to John Baez, Marcelo Fiore, Weiwei Pan, Sam Staton and Chenchang Zhu for useful comments and discussions. The graphics in this paper were produced using \texttt{TikZ} and \texttt{xypic}.

\section{Groupoidification of the Heisenberg algebra}
\label{sec:groupoidification}

\subsection{Spans of groupoids}

Our combinatorial representation of the categorified Heisenberg algebra will be given in terms of \emph{spans of groupoids}. A groupoid is a category in which all morphisms are invertible, and a span of groupoids from \cat A to \cat B is a diagram of the following form, where $\cat A$, $\cat B$ and $\cat X$ are groupoids and $F$ and $G$ are functors:
\begin{equation}
\begin{aligned}
\begin{tikzpicture}[xscale=2, yscale=1.5]
\node (T) at (0,0) {$\cat X$};
\node (L) at (-1,-1) {$\cat B$};
\node (R) at (1,-1) {$\cat A$};
\draw [->] (T) to node [auto, swap] {$G$} (L);
\draw [->] (T) to node [auto] {$F$} (R);
\end{tikzpicture}
\end{aligned}
\end{equation}
A span like this forms a 1-morphism of type $\cat A \to \cat B$ in the monoidal 2-category \cat{Span(Gpd)}. The symmetry in the definition of a span means that, for any span $\cat B \sxlto G \cat X \sxto F \cat A$ as above, we can define its \emph{converse} $(\cat B \sxlto G \cat X \sxto F \cat A) ^\dag$ as the span $\cat A \sxlto F \cat X \sxto G \cat B$.  We will note that this converse is in fact an adjoint.

If we think of the objects of \cat A and \cat B as being the states of some physical system, then any $x \in \Ob(\cat X)$ can be interpreted as a `history' relating the state $F(x) \in \Ob(\cat A)$ to the state $G(x) \in \Ob(\cat B)$. Because we are working with groupoids, our states and histories come equipped with symmetry groups, and the functors $F$ and $G$ show how the symmetries of histories are mapped to symmetries of states.

The combinatorial interpretation that we will explore for these spans arises from the fundamental role played here by the groupoid of finite sets. The Fock space is represented by the groupoid ${\FSN}$ and the annihilation and creation operators $A$ and $A ^\dag$ by the following spans of groupoids:
\begin{equation}
\label{eq:loweringraising}
\begin{aligned}
\begin{tikzpicture}[xscale=2, yscale=1.5]
\node (T) at (0,0) {\FSN};
\node (L) at (-1,-1) {\FSN};
\node (R) at (1,-1) {\FSN};
\draw [->] (T) to node [auto, swap] {\id} (L);
\draw [->] (T) to node [auto] {$+1$} (R);
\end{tikzpicture}
\end{aligned}
\hspace{70pt}
\begin{aligned}
\begin{tikzpicture}[xscale=2, yscale=1.5]
\node (T) at (0,0) {\FSN};
\node (L) at (-1,-1) {\FSN};
\node (R) at (1,-1) {\FSN};
\draw [->] (T) to node [auto, swap] {$+1$} (L);
\draw [->] (T) to node [auto] {\id} (R);
\end{tikzpicture}
\end{aligned}
\end{equation}
Here $\FSN$ is the groupoid of finite sets and bijections, and $\FSN \sxto {+1} \FSN$ is the functor taking the disjoint union with the one-element set.

A different perspective on $\FSN$ is to see it as the \textit{free symmetric monoidal groupoid} on the trivial groupoid \cat 1 with one object and one morphism. The functor $\FSN \sxto{+1} \FSN$ then arises by taking the tensor product with the generating object. In fact, this free symmetric monoidal perspective is fundamental, and allows for the construction of a representation of the categorified Heisenberg algebra in completely abstract terms. We develop this perspective in detail in the companion article~\cite{mortonvicary}.

\subsection{Combinatorics}
\label{sec:combinatorics}

Spans of groupoids can be used to reason about the combinatorics of structures on finite sets. The theory of \emph{stuff types}~\cite{finfeyn} has been developed to describe how this works, generalizing Joyal's theory of \emph{structure types}~\cite{joyal1986}. Given some particular structure of interest constructed from a finite set, the groupoid \cat G of \textit{models} of this structure can be built, with morphisms given by symmetries that relate one model to another. This comes equipped with a functor $\cat G \sxto F \cat \FSN$, called a \textit{stuff type}, where \FSN{} is the groupoid of finite sets and bijections. This functor assigns to each model its underlying set, and to symmetries of models the induced bijections on the underlying sets. We can interpret any abstract stuff type as describing some combinatorial structure; in general, a finite set equipped with some structure which satisfies some property. A stuff type with good properties can be \textit{degroupoidified}, a form of \emph{linearization}, giving rise to a generating function for the structure in the ordinary sense.

Since every groupoid has a unique functor to the 1-element set, stuff types are precisely spans of the form
\begin{equation}
\begin{aligned}
\begin{tikzpicture}[xscale=2, yscale=1.5]
\node (T) at (0,0) {$\cat G$};
\node (L) at (-1,-1) {\FSN};
\node (R) at (1,-1) {$\cat 1\,,$};
\draw [->] (T) to node [auto, swap] {$F$} (L);
\draw [->] (T) to node [auto] {} (R);
\end{tikzpicture}
\end{aligned}
\end{equation}
which are morphisms of type $\cat 1 \to \cat S$ in \cat{Span(Gpd)}. A \emph{stuff operator} is a span of groupoids of type $\cat S \to \cat S$:
\begin{equation}
\begin{aligned}
\begin{tikzpicture}[xscale=2, yscale=1.5]
\node (T) at (0,0) {$\cat H$};
\node (L) at (-1,-1) {\FSN};
\node (R) at (1,-1) {$\FSN$};
\draw [->] (T) to node [auto, swap] {$K$} (L);
\draw [->] (T) to node [auto] {$J$} (R);
\end{tikzpicture}
\end{aligned}
\end{equation}
A stuff operator thus contains a single groupoid \cat H of models, but two potentially different descriptions $K$ and $J$ of underlying sets; it tells us  how the same structure can be built in different ways. In the physical interpretation we have mentioned, these models --- the objects in the middle groupoid --- are seen as \textit{histories}, and the two underlying sets are the starting and ending configurations of these histories. Stuff operators can act on stuff types by composition in \cat{Span(Gpd)} as described in Section~\ref{sec:composingspans}, producing new stuff types whose structures are composites of those from the original stuff type with the histories from the stuff operator.

For example, consider the stuff type $\FSN \sxto{+1} \FSN$ which takes the union of each set with a chosen 1-element set. This represents the structure `finite sets with a chosen element', with symmetries given by bijections that leave the chosen element fixed. Then our annihilation operator $A$ defined in~\eqref{eq:loweringraising} is a stuff operator which treats an $n$-element set in the middle copy of \FSN{} in two different ways: by the right leg, as an $n\!+\!1$--element set with a chosen element; and by the left leg, simply as an $n$-element set. It can be thought of as a `rule' for how to consider an $n\!+\!1$--element set as an $n$-element set --- or more simply, as a way to \emph{remove an element} from a finite set.

\subsection{Composing spans}
\label{sec:composingspans}

Given spans of groupoids \mbox{$\cat B \sxlto G \cat X \sxto F \cat A$} and \mbox{$\cat C \sxlto {K} \cat Y \sxto {J} \cat B$}, we can compose them by constructing a \emph{weak pullback} groupoid $(J \downarrow G)$:
\begin{equation}
\label{eq:composespans}
\begin{aligned}
\begin{tikzpicture}[xscale=2, yscale=1.5]
\node (Y) at (0,0) {$\cat Y$};
\node (C) at (-1,-1) {\cat C};
\node (B) at (1,-1) {$\cat B$};
\node (A) at (3,-1) {$\cat A$};
\node (X) at (2,0) {$\cat X$};
\node (P) at (1,1) {$ {(J \downarrow G)}$};
\draw [->] (X) to node [auto] {$F$} (A);
\draw [->] (X) to node [auto, swap] {$G$} (B);
\draw [->] (Y) to node [auto] {$J$} (B);
\draw [->] (Y) to node [auto, swap] {$K$} (C);
\draw [->] (P) to node [auto, swap] {$P_{\cat Y}$} (Y);
\draw [->] (P) to node [auto] {$P_{\cat X}$} (X);
\draw [->, double] (Y) to
    node [auto] {$\alpha$}
    node [auto, swap] {$\simeq$}
        (X);
\draw [->] (P)
    to [out=180, in=90]
    node [pos=0.5, auto, swap]
    {$K \circ P_\cat Y$}
        (C);
\draw [->] (P)
    to [out=0, in=90]
    node [pos=0.5, auto]
    {$F \circ P_\cat X$}
        (A);
\end{tikzpicture}
\end{aligned}
\end{equation}
This groupoid $(J \downarrow G)$ is equipped with functors \mbox{$(J \downarrow G) \sxto {P _\cat Y} \cat Y$} and \mbox{$(J \downarrow G) \sxto {P _\cat X} \cat X$}, and a natural isomorphism $J \circ P _\cat Y \sxto \alpha G \circ P _\cat Y$. It is defined to satisfy a universal property: for any groupoid $\cat Z$ equipped with functors $\cat Z \sxto{Z_\cat Y} \cat Y$ and $\cat Z \sxto{Z _\cat X} \cat X$ and a natural isomorphism $J \circ Z_\cat X \sxto \zeta G \circ Z _\cat Y$, there must exist a functor $\cat Z \sxto L \cat (J \downarrow G)$, unique up to isomorphism, such that $L \circ \alpha = \zeta$. The resulting composite span is defined to be \mbox{$\cat C \sxlto {K \circ P _\cat Y} (J \downarrow G) \sxto {F \circ P_ \cat X} \cat A$}.

Based on the universal property described above, a standard construction for $(J \downarrow G)$ is the following groupoid:
\begin{itemize}
\item \textbf{Objects} are triples $\big(x \in \Ob(\cat X) , y \in \Ob (\cat Y),  G(x) \sxto f J(y) \big)$.
\item \textbf{Morphisms} $(x_1, y_1, f_1) \to (x_2, y_2, f_2)$ are pairs of morphisms $x_1 \sxto a x_2$ and $y_1 \sxto b y_2$ satisfying the following commuting diagram:
\begin{equation}
\begin{aligned}
\begin{tikzpicture}
\node (G1) at (0,0) {$G(x_1)$};
\node (J1) at (2,0) {$J(y_1)$};
\node (G2) at (0,-2) {$G(x_2)$};
\node (J2) at (2,-2) {$J(y_2)$};
\draw [->] (G1) to node [auto, swap] {$G(a)$} (G2);
\draw [->] (G2) to node [auto, swap] {$f_2$} (J2);
\draw [->] (G1) to node [auto] {$f_1$} (J1);
\draw [->] (J1) to node [auto] {$J(b)$} (J2);
\end{tikzpicture}
\end{aligned}
\end{equation}
\end{itemize}
This construction is essentially a weak form of the fibered product, where instead of taking pairs $(x,y)$ whose images in $\cat{B}$ agree, we choose a specific isomorphism between them.  We can unpack what this means for stuff types, for which $\cat A = \cat B = \cat C = \FSN$, in a concise way: an object in the weak pullback is a pair of models in $\cat X$ and $\cat Y$ equipped with a bijection of underlying sets; and a morphism is a pair of symmetries of models which induce the same bijections on the underlying sets.

\subsection{The categorified commutation relation}

The categorified form of the commutation relation~\eqref{eq:ccr} is an isomorphism of spans of the following form:
\begin{equation}
\label{eq:ccriso}
A \circ A ^\dag \simeq A ^\dag \circ A \oplus \id _\FSN
\end{equation}
The symbol `$\oplus$' represents the \emph{direct sum} of spans, which is formed from the disjoint union of groupoids of histories.

We begin with an intuitive argument for why this isomorphism should exist. We saw above that the histories of the span $A ^\dag$ represent all the ways to add an  element to a finite set, and the histories of $A$ represent all the ways to remove an element. The histories of $A \circ A ^\dag$ therefore represent all the ways to add an element $x$, and then remove an element $y$. Such histories can be divided into two distinct classes: those for which $x \neq y$, and those for which $x = y$. Restricting to the first case, one might as well remove $y$ before adding $x$, and so we have a bijection to the histories for the span $A ^\dag \circ A$; in the second case, the set remains unchanged, corresponding to the identity span. These two cases give the two terms on the right-hand side, and the explicit case-by-case bijection we have constructed gives the isomorphism of spans.

We now verify this intuition with direct calculation. The composite $A \circ A^{\dagger}$ is the following span:
\begin{equation}
\label{eq:p1p1span}
\begin{aligned}
\begin{tikzpicture}[xscale=2, yscale=1.5]
\node (T) at (0,0) {$(+1 \downarrow +1)$};
\node (L) at (-1,-1) {\FSN};
\node (R) at (1,-1) {\FSN};
\draw [->] (T) to node [auto, swap] {$P$} (L);
\draw [->] (T) to node [auto] {$Q$} (R);
\end{tikzpicture}
\end{aligned}
\end{equation}

\noindent
The groupoid \mbox{$(+1 \downarrow +1)$} has objects which are triples $(s,t,\alpha)$, where $s$ and $t$ are sets and $s+1 \sxto \alpha t+1$ is an isomorphism. The projection maps $P$ and $Q$ have the actions $P(s,t,\alpha) = s$ and $Q(s,t,\alpha) = t$ on objects.

A morphism $(s,t,\alpha) \to (s',t',\alpha')$ is a pair of isomorphisms $s \xrightarrow{\smash{\sigma}} s'$ and $t \xrightarrow{\smash{\tau}} t'$ such that $\alpha' \circ (\sigma+1) = (\tau+1) \circ \alpha$, with the projection maps $P$ and $Q$ taking this to $\sigma$ and $\tau$ respectively. We can visualize this condition in the following way:
\def\shift{8cm}
\def\redcolor{red!50}
\begin{equation}
\begin{aligned}
\begin{tikzpicture}[scale=0.5]
\node (a1) [blob] at (0,-0.5) {};
\node (b1) [blob] at (1.5,1) {};
\node (c1) [blob, \redcolor] at (2.5,2) {};
\node [rotate=-45] at (0.9,0.4) {$\vdots$};
\draw [rotate=-45, thick] (0.3,1.5) ellipse (0.8 and 2.5);
\begin{scope}[xshift=\shift]
\node (a2) [blob] at (0,-0.5) {};
\node (b2) [blob] at (1.5,1) {};
\node (c2) [blob, \redcolor] at (2.5,2) {};
\node [rotate=-45] at (0.9,0.4) {$\vdots$};
\draw [rotate=-45, thick] (0.3,1.5) ellipse (0.8 and 2.5);
\end{scope}
\begin{scope}[yshift=-\shift]
\node (a3) [blob] at (0,-0.5) {};
\node (b3) [blob] at (1.5,1) {};
\node (c3) [blob, \redcolor] at (2.5,2) {};
\node [rotate=-45] at (0.9,0.4) {$\vdots$};
\draw [rotate=-45, thick] (0.3,1.5) ellipse (0.8 and 2.5);
\end{scope}
\begin{scope}[yshift=-\shift,xshift=\shift]
\node (a4) [blob] at (0,-0.5) {};
\node (b4) [blob] at (1.5,1) {};
\node (c4) [blob, \redcolor] at (2.5,2) {};
\node [rotate=-45] at (0.9,0.4) {$\vdots$};
\draw [rotate=-45, thick] (0.3,1.5) ellipse (0.8 and 2.5);
\end{scope}
\draw (a1) to (a2);
\draw (b1) to (b2);
\draw (c1) to (c2);
\draw (a1) to (a3);
\draw (b1) to (b3);
\draw (c1) to (c3);
\draw (a4) to (a2);
\draw (b4) to (b2);
\draw (c4) to (c2);
\draw (a4) to (a3);
\draw (b4) to (b3);
\draw (c4) to (c3);
\node
    [minimum width=1.2cm, minimum height=1.8cm, draw, fill=white]
    at (5.25,0.75) {$\alpha$};
\node
    [minimum width=1.2cm, minimum height=1.8cm, draw, fill=white]
    at (5.25,-7.25) {$\alpha'$};
\node
    [minimum width=1.2cm, minimum height=1.3cm, draw, fill=white, rotate=-90]
    at (0.75,-3.25) {$\sigma$};
\node
    [minimum width=1.2cm, minimum height=1.3cm, draw, fill=white, rotate=-90]
    at (8.75,-3.25) {$\tau$};
\node [anchor=east] at (-0.2,1.0) {$s+1$};
\node [anchor=west] at (11.2,1.0) {$t+1$};
\node [anchor=east] at (-0.2,-7) {$s'+1$};
\node [anchor=west] at (11.2,-7) {$t'+1$};
\draw [->, thick] (3.5,-1.5)
    to [out=right, in=up] (7.0,-4.5);
\draw [->, thick] (3,-2)
    to [out=down, in=left] (6.5,-5);
\end{tikzpicture}
\end{aligned}
\end{equation}
The extra element in each set is drawn in red. We show explicitly in this diagram how the permutations $\alpha$, $\alpha'$, $\sigma$ and $\tau$ act on the different parts of each set. For the condition to be satisfied, the two composites from the top-left to the bottom-right must be equal.

With this diagram to aid us, we can draw some interesting conclusions. Firstly, for $(s,t,\alpha)$ and $(s',t',\alpha')$ to be isomorphic, then $\alpha$ must fix the extra element iff $\alpha'$ does --- and this is sufficient as long as $s \simeq s'$ and $t \simeq t'$. So for each isomorphism class of set $S$, there are exactly two isomorphism classes of object in its preimage under $P$ or $Q$ in $(+1 \downarrow +1)$. Secondly, if $\alpha$ and $\alpha'$ fix the extra element, then for every $\sigma$ there exists some $\tau$ making the diagram commute, so there is an $S_n$-worth of morphisms between any two such objects for $n=|S|$. Thirdly, if neither $\alpha$ nor $\alpha'$ fix the extra element, then a given $\tau$ can only be part of a commuting diagram if $\tau(\alpha(1))=\alpha'(1)$; in this case we can define $\alpha$ as the restriction to $S$ of the composite $\alpha' {}^{-1} \circ (\tau + 1) \circ \alpha$. This means that, in this case, there is a $S_{n-1}$ worth of morphisms between any two such objects for $n = |S|$. 

This gives us a complete understanding of the isomorphism classes of object in \mbox{$(+1 \downarrow +1)$} and their symmetries, which are the union of the entries in the following table:
\begin{equation}
\label{eq:table1}
\begin{aligned}
\begin{tikzpicture}
\node [align=right, font=\bf] at (0,0) {Extra element\\[-5pt]fixed};
\node [align=right, font =\bf] at (0,1) {Extra element\\[-5pt]not fixed};
\foreach\x in {0,1,2,3}
{
    \node at (2cm + \x cm,2) {\textbf{\x}};
    \node at (2cm + \x cm,0) {$S_{\x}$};
}
\foreach\x in {0,1,2}
    \node at (3cm + \x cm,1) {$S_{\x}$};
\node at (6cm,0) {$\cdots$};
\node at (6cm,1) {$\cdots$};
\node at (6cm,2) {$\cdots$};
\node at (8cm,0) {$\cdots$};
\node at (8cm,1) {$\cdots$};
\node at (8cm,2) {$\cdots$};
\node at (7cm,2) {$\boldsymbol n$};
\node at (7cm,1) {$S_{n-1}$};
\node at (7cm,0) {$S_{n}$};
\end{tikzpicture}
\end{aligned}
\end{equation}
The column headings indicate the cardinality of the set to which each object is projected, under the projection maps $P$ and $Q$. As a result, we see that the span~\eqref{eq:p1p1span} can be written in the following way:
\begin{equation}
\label{eq:span-AAs}
\begin{aligned}
\begin{tikzpicture}[xscale=2, yscale=1.5]
\node (T) at (0,0) {$\FSN \cup \FSN$};
\node (L) at (-1,-1) {\FSN};
\node (R) at (1,-1) {\FSN};
\draw [->] (T) to node [auto, swap] {$(\id _\FSN, +1) $} (L);
\draw [->] (T) to node [auto] {$(\id_\FSN, +1)$} (R);
\end{tikzpicture}
\end{aligned}
\end{equation}
Here $\FSN \cup \FSN$ represents the disjoint union of two copies of $\FSN$, and $(\id_\FSN, +1)$ represents the functor which acts as the identity on the first copy of $\FSN$ and as $+1$ on the second. We can consider this to be the union of spans $\big( \FSN \sxlto{\id_\FSN} \FSN \sxto{\id_\FSN} \FSN \big)$ and $\big( \FSN \sxlto{+1} \FSN \sxto{+1} \FSN \big)$, which gives us the isomorphism~\eqref{eq:ccriso} we are seeking.

\subsection{Forming a 2-category}
\label{sec:forming2cat}

So far our construction essentially follows that given in \cite{mstuff}, but we now need to go further and use the 2\-categorical structure of $\cat{Span(Gpd)}$.  This requires a notion of morphism between spans. For this, we define a \emph{span of spans} of type \mbox{$(\cat B \sxlto G \cat X \sxto F \cat A) \to (\cat B \sxlto J \cat Y \sxto K \cat A)$}  is a span \mbox{$\cat X \sxlto S \cat Z \sxto T \cat Y$} equipped with natural transformations $G \circ S \sxto \mu J \circ T$ and $F \circ S \sxto \nu K \circ T$, as indicated by the following diagram:
\begin{equation}
\label{eq:2spandiag}
\begin{aligned}
\begin{tikzpicture}[xscale=2, yscale=1.5]
\node (X) at (0,0) {$\cat X$};
\node (Y) at (0,-2) {$\cat Y$};
\node (Z) at (0,-1) {$\cat Z$};
\node (B) at (-1,-1) {$\cat B$};
\node (A) at (1,-1) {$\cat A$};
\draw [->] (X) to node [auto, swap] {$G$} (B);
\draw [->] (X) to node [auto] {$F$} (A);
\draw [->] (Y) to node [auto, swap] {$K$} (A);
\draw [->] (Y) to node [auto] {$J$} (B);
\draw [->] (Z) to node [auto, swap] {$S$} (X);
\draw [->] (Z) to node [auto] {$T$} (Y);
\draw [->, double] (0.4,-0.7) to
    node [right, pos=0.5] {$\nu$}
    (0.4,-1.3);
\draw [->, double] (-0.4,-0.7) to
    node [right, pos=0.5] {$\mu$}
    (-0.4,-1.3);
\end{tikzpicture}
\end{aligned}
\end{equation}
We say that two such spans of spans $(\cat X \sxlto S \cat Z \sxto T \cat Y, \mu, \nu)$ and $(\cat X \sxlto {S'} \cat {Z'} \sxto {T'} \cat Y, \mu', \nu')$ are \textit{equivalent} when there is an equivalence of groupoids $\cat Z \sxto{U} \cat{Z'}$, and natural transformations $S \sxto{\sigma} S' \circ U$ and $T \sxto{\tau} T' \circ U$, such that the following pasted composites give $\mu$ and $\nu$ respectively:
\begin{equation}
\label{eq:equivalencerelation}
\begin{aligned}
\begin{tikzpicture}[xscale=2, yscale=1.5]
\node (X) at (0,0) {$\cat X$};
\node (Y) at (0,-2) {$\cat Y$};
\node (Z) at (0.3,-1) {$\cat {Z}$};
\node (Z') at (-0.3,-1) {$\cat {Z'}$};
\node (A) at (-1.2,-1) {$\cat B$};
\draw [->] (X) to [out=180, in=60] node [auto, swap] {$G$} (A);
\draw [->] (Y) to [out=180, in=-60] node [auto] {$J$} (A);
\draw [->] (Z) to [out=90, in=-50] node [auto, swap, pos=0.2] {$S$} (X);
\draw [->] (Z) to [out=-90, in=50] node [auto, pos=0.2] {$T$} (Y);
\draw [->] (Z') to [out=90, in=-130] node [auto, pos=0.2, inner sep=2pt] {$S'$} (X);
\draw [->] (Z') to [out=-90, in=130] node [auto, swap, pos=0.2, inner sep=2pt] {$T'$} (Y);
\draw [->, double] (-0.7,-0.7) to
    node [right, pos=0.5] {$\mu'$}
    (-0.7,-1.3);
\draw [->, double] (0.15,-0.5) to
    node [auto, swap, pos=0.3, inner sep=1pt] {$\sigma$}
    (-0.2,-0.75);
\draw [->, double] (0.15,-1.5) to
    node [auto, pos=0.3, inner sep=2pt] {$\tau$}
    (-0.2,-1.25);
\draw [->] (Z) to node [auto, swap, inner sep=2pt] {$U$} (Z');
\end{tikzpicture}
\end{aligned}
\hspace{60pt}
\begin{aligned}
\begin{tikzpicture}[xscale=2, yscale=1.5]
\node (X) at (0,0) {$\cat X$};
\node (Y) at (0,-2) {$\cat Y$};
\node (Z) at (-0.3,-1) {$\cat Z$};
\node (Z') at (0.3,-1) {$\cat {Z'}$};
\node (A) at (1.2,-1) {$\cat A$};
\draw [->] (X) to [out=0, in=120] node [auto] {$F$} (A);
\draw [->] (Y) to [out=0, in=-120] node [auto, swap] {$K$} (A);
\draw [->] (Z) to [out=90, in=-130] node [auto, pos=0.2] {$S$} (X);
\draw [->] (Z) to [out=-90, in=130] node [auto, swap, pos=0.2] {$T$} (Y);
\draw [->] (Z') to [out=90, in=-50] node [auto, swap, pos=0.2] {$S'$} (X);
\draw [->] (Z') to [out=-90, in=50] node [auto, pos=0.2] {$T'$} (Y);
\draw [->, double] (0.7,-0.7) to
    node [right, pos=0.5] {$\nu'$}
    (0.7,-1.3);
\draw [->, double] (-0.15,-0.5) to
    node [auto, pos=0.4, inner sep=1pt] {$\sigma$}
    (0.2,-0.75);
\draw [->, double] (-0.15,-1.5) to
    node [auto, swap, pos=0.3, inner sep=2pt] {$\tau$}
    (0.2,-1.25);
\draw [->] (Z) to node [auto, inner sep=2pt] {$U$} (Z');
\end{tikzpicture}
\end{aligned}
\end{equation}

We use this to form a 2-category $\cat{Span(Gpd)}$ of spans of
groupoids. Since this has been described in some detail elsewhere
\cite{gpd2vect}, in the case of finite groupoids, we will avoid much
formal discussion of it here. However, we wish to use groupoids which
are discrete but not finite, and there is an extra assumption which
must be added for this to have a good linear representation theory. This is addressed in work
on groupoidification \cite{bhw-gpd} for the 1-category case, where it
is noted that groupoids and spans must be \textit{tame}.

A groupoid is \textit{tame} if it satisfies three properties. First, that it is
essentially small: that is, equivalent to one with a set of
objects, rather than a proper class. Second, that it is locally
finite: that is, all Hom-sets are finite. Third, its ``groupoid
cardinality'' should be finite: this is the sum over equivalence
classes of objects $x$ of $|\mathrm{Aut}(x)|^{-1}$. We will almost exclusively
use the groupoid of finite sets and bijections, which is indeed tame:
the objects are determined, up to isomorphism, by an integer $n$; the
automorphism groups are isomorphic to the permutation groups $\Sn$,
which are finite; the groupoid cardinality is therefore $\sum_n
\frac{1}{n!} = e$, which is finite.

A span is tame when the joint preimage of any pair of
objects in the source and target groupoid is tame, and also
for any object in the target groupoid, only finitely many choices of
object in the source groupoid gives a joint preimage which is
nonempty. Note that tameness of a span is an asymmetrical condition, so the opposite notion of cotame span is also interesting. These conditions of tameness and cotameness can be applied equally to spans of spans as
in our 2-categorical situation.

We use these definitions to build $\SiiG$ in the following way:
\begin{itemize}
\item \textbf{Objects} are tame groupoids.
\item \textbf{Morphisms} are spans of groupoids which are tame and cotame, with composition defined by weak pullback.
\item \textbf{2-Morphisms} are equivalence classes of spans of spans which are tame and cotame.
\end{itemize}

\noindent
This 2-category will be the formal setting for our results. The symmetric monoidal 2\-category structure of \cat{Span(Gpd)} has been described by Hoffnung and Stay~\cite{hoffnungspans,stay-compact}, although in the simpler situation where which the 2-morphisms are span maps rather than spans of spans. We conjecture that the monoidal 2-category structure extends to our case.

The 2-category \cat{Span(Gpd)} has strong duality properties. We have already defined a notion of converse for 1\-morphisms. In addition, for any 2\-morphism
\[
(\cat B \sxlto G \cat X \sxto F \cat A) \sxto {(\cat X \sxlto{S} \cat Z \sxto T \cat Y, \mu,\nu)} (\cat B \sxlto J \cat Y \sxto K \cat A),
\]
we can define its \textit{converse} $(\cat X \sxlto S \cat Z \sxto T \cat Y, \mu,\nu) ^\dag$ as the following span of spans:
\vspace{3pt}
\[
(\cat B \sxlto J \cat Y \sxto K \cat A)
\sxto{(\cat Y \sxlto T \cat Z \sxto S \cat X, \mu ^{-1}, \nu ^{-1} )}
(\cat B \sxlto G \cat X \sxto F \cat A).
\]

\subsection{The commutation relation}

Based on the calculations in the previous section, we can define the following spans of spans, called $i _{\id_\FSN}$ and $i _{A \circ A ^\dag}$:
\begin{equation}
\begin{aligned}
\begin{tikzpicture}[xscale=2, yscale=1.5]
\node (X) at (0,0) {$\FSN \cup \FSN$};
\node (Y) at (0,2) {$\FSN$};
\node (Z) at (0,1) {$\FSN$};
\node (B) at (-1,1) {$\FSN$};
\node (A) at (1,1) {$\FSN$};
\draw [->] (X) to node [auto] {$(\id _\FSN,+1)$} (B);
\draw [->] (X) to node [auto, swap] {$(\id _\FSN,+1)$} (A);
\draw [->] (Y) to node [auto] {$\id_\FSN$} (A);
\draw [->] (Y) to node [auto, swap] {$\id_\FSN$} (B);
\draw [->] (Z) to node [auto] {$I_1$} (X);
\draw [->] (Z) to node [auto, swap] {$\id_\FSN$} (Y);
\draw [<-, double] (0.4,0.7) to
    node [right, pos=0.5] {$\id$}
    (0.4,1.3);
\draw [<-, double] (-0.4,0.7) to
    node [right, pos=0.5] {$\id$}
    (-0.4,1.3);
\node at (0,-0.5) {$\id_\FSN \sxto{i_{\id_\FSN}} A \circ A ^\dag$};
\end{tikzpicture}
\end{aligned}
\hspace{30pt}
\begin{aligned}
\begin{tikzpicture}[xscale=2, yscale=1.5]
\node (X) at (0,0) {$\FSN \cup \FSN$};
\node (Y) at (0,2) {$\FSN$};
\node (Z) at (0,1) {$\FSN$};
\node (B) at (-1,1) {$\FSN$};
\node (A) at (1,1) {$\FSN$};
\draw [->] (X) to node [auto] {$(\id _\FSN,+1)$} (B);
\draw [->] (X) to node [auto, swap] {$(\id _\FSN,+1)$} (A);
\draw [->] (Y) to node [auto] {$+1$} (A);
\draw [->] (Y) to node [auto, swap] {$+1$} (B);
\draw [->] (Z) to node [auto] {$I_2$} (X);
\draw [->] (Z) to node [auto, swap] {$\id _\FSN$} (Y);
\draw [<-, double] (0.4,0.7) to
    node [right, pos=0.5] {$\id$}
    (0.4,1.3);
\draw [<-, double] (-0.4,0.7) to
    node [right, pos=0.5] {$\id$}
    (-0.4,1.3);
\node at (0,-0.5) {$A^\dag \circ A \sxto{i_{A^\dag \circ A}} A \circ A ^\dag$};
\end{tikzpicture}
\end{aligned}
\end{equation}
Here, $I_1$ and $I_2$ are functors embedding the first and second copy of $\FSN$ respectively.

We look more closely at the definition of $i _{\id _\FSN}$, looking `inside' the groupoids for the case of histories acting on the 2-element set:
\begin{equation}\label{diag:i-id-s}
\begin{aligned}
\includegraphics{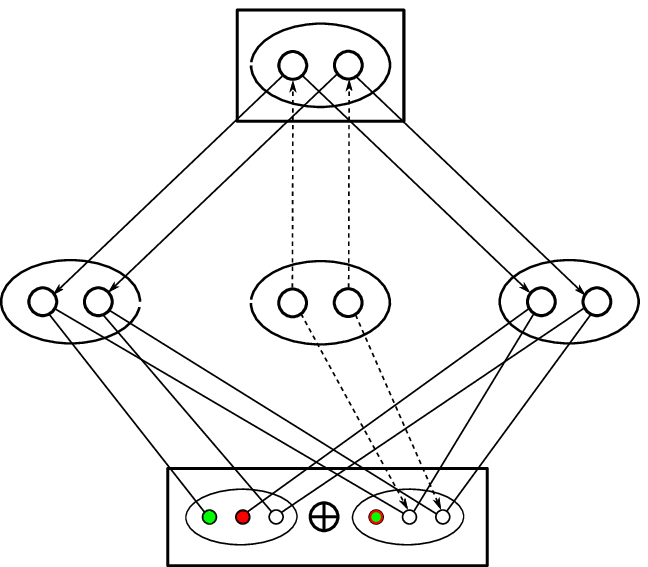}
\end{aligned}
\end{equation}
Across the top of this picture, we see one object of the middle $\FSN$, namely the two-element set.  The identity span maps this object down to the two-element set on either side.  On the bottom, though, we see the two objects in $(+1 \downarrow +1)$ which map down to two-element sets.  Each is a 3-element set, but as in the table (\ref{eq:table1}), we see that the marked elements --- the element added and the element removed --- are either the same or different.  The symmetry groups noted in \eqref{eq:table1} are those which permute the unmarked elements.  The span of spans shows how objects in the identity relate to particular objects in $(+1 \downarrow +1)$.  This determines an inclusion.

We use these ideas to demonstrate that the categorified commutation relation holds.
\begin{lemma}
\label{lem:ccrbiproduct}
There is an isomorphism of spans of groupoids
\begin{equation}
\label{eq:formalcommutationrelation}
A \circ A ^\dag \simeq (A ^\dag \circ A) \oplus \id _\FSN,
\end{equation}
where $\oplus$ represents the $\dag$-biproduct of spans of groupoids.
\end{lemma}
\begin{proof}
The $\dagger$-biproduct is witnessed by the following equations involving the injection 2\-morphisms $i _{\id _\FSN}$, $i _{A ^\dag  \circ A}$ and their converses:
\begin{align}
\label{eq:biproduct1}
\id _{\id _\FSN}
&=
(i _{\id _\FSN}) ^\dag \circ i _{\id _\FSN}
\\
\label{eq:biproduct2}
0 _{A^\dag \circ A, \id _\FSN}
&=
(i _{\id _\FSN}) ^\dag \circ i _{A ^\dag \circ A}
\\
\label{eq:biproduct3}
0 _{\id_\FSN, A ^\dag \circ A}
&=
(i _{A ^\dag \circ A}) ^\dag \circ i _{\id _\FSN}
\\
\label{eq:biproduct4}
\id _{A ^\dag \circ A}
&=
(i _{A ^\dag \circ A})  ^\dag \circ i _{A ^\dag \circ A}
\\
\label{eq:biproduct5}
\id _{A \circ A ^\dag}
&=
i _{\id _\FSN} \circ (i _{\id _\FSN}) ^\dag + i _{A^\dag \circ A} \circ (i _{A ^\dag \circ A}) ^\dag
\end{align}
We will see that these injection and projection maps, which characterize the $\dagger$-biproduct, are related to the adjointness of the spans $A$ and $A^{\dagger}$. Correctness of these equations follows from the combinatorial interpretations of these spans of spans, which we develop below.
\end{proof}

\subsection{Graphical notation}

We now introduce a graphical notation for certain spans of type $\FSN \sxlto{} \cat X \sxto{} \FSN$, and for the 2\-morphisms going between them in \cat{Span(Gpd)}. This is the notation of Khovanov~\cite{khovanov}, and is an application of the standard graphical calculus for morphisms in a monoidal category~\cite{selinger-diagrams}. Creation operators $A^\dag$ and annihilation operators $A$ are represented as vertical lines with upwards and downwards orientation, respectively:
\begin{equation}
\begin{aligned}
\begin{tikzpicture}
\draw [arrow=0.5] (0,0) node [below] {$A^\dag$} to (0,1);
\end{tikzpicture}
\end{aligned}
\hspace{100pt}
\begin{aligned}
\begin{tikzpicture}
\draw [reverse arrow=0.5] (0,0) node [below] {$A$} to (0,1);
\end{tikzpicture}
\end{aligned}
\end{equation}
Composition of operators is represented by horizontal juxtaposition. The identity span $\FSN \sxlto{\id _\FSN} \FSN \sxto{\id _\FSN} \FSN$ is represented by the empty diagram.

We denote 2-morphisms with string diagrams. In particular, our 2-morphisms
$i _{\id _\FSN}$, $(i _{\id _\FSN}) ^\dag$, $i _{A ^\dag \circ A}$ and $(i _{A ^\dag \circ A}) ^\dag$ have the following representations:
\renewcommand{\centerdia}[1]{\makebox[2cm]{\ensuremath{#1}}}
\begin{align}
\label{eq:basiccomponents}
\centerdia{\begin{aligned}
\begin{tikzpicture}
\draw [white] (0,0) to (0,-1.5);
\draw (0,0) to (0,-0.2) to [out=down, in=down, looseness=2] (1,-0.2) to (1,0);
\draw [<-] (0,-0.2) to (0,0);
\draw [-<] (1,0) to (1,-0.2);
\end{tikzpicture}
\end{aligned}}
&&
\centerdia{\begin{aligned}
\begin{tikzpicture}
\draw [white] (0,-0.2) to (0,1.3);
\draw (0,-0.2) to (0,0) to [out=up, in=up, looseness=2] (1,0) to (1,-0.2);
\draw [-<] (0,-0.2) to (0,0);
\draw [<-] (1,0) to (1,-0.2);
\end{tikzpicture}
\end{aligned}}
&&
\centerdia{\begin{aligned}
\begin{tikzpicture}
\draw [arrow=0.1, arrow=0.9] (0,0) to [out=up, in=down] (1,1.5);
\draw [reverse arrow=0.1, reverse arrow=0.9] (1,0) to [out=up, in=down] (0,1.5);
\end{tikzpicture}
\end{aligned}}
&&
\centerdia{\begin{aligned}
\begin{tikzpicture}
\draw [reverse arrow=0.1, reverse arrow=0.9] (0,0) to [out=up, in=down] (1,1.5);
\draw [arrow=0.1, arrow=0.9] (1,0) to [out=up, in=down] (0,1.5);
\end{tikzpicture}
\end{aligned}}
\\
\nonumber
\centerdia{i _{\id _\FSN}}
&&
\centerdia{(i _{\id _\FSN}) ^\dag}
&&
\centerdia{i _{A ^\dag \circ A}}
&&
\centerdia{(i _{A ^\dag \circ A}) ^\dag}
\end{align}
Using this notation, the biproduct equations~(\ref{eq:biproduct1}\_\ref{eq:biproduct5}) have the following representation:
\noindent\begin{tabular*}{\textwidth}{@{\extracolsep{\fill}}E*{3}{>{$}c<{$}}E}
&
\begin{aligned}
\begin{tikzpicture}
\draw [reverse arrow=0, reverse arrow=0.5] (0,0)
    to [out=up, in=up, looseness=2] (1,0)
    to [out=down, in=down, looseness=2] (0,0);
\end{tikzpicture}
\end{aligned}
\,\,=\,\,
\id _{\id _\FSN}
&
\begin{aligned}
\begin{tikzpicture}
\draw [white] (0,-0.2) to (1,-0.2);
\draw [arrow=0.05, arrow=0.37, arrow=0.64, arrow=0.96] (0,0)
    to [out=up, in=down] (1,1.5)
    to [out=up, in=up, looseness=2] (0,1.5)
    to [out=down, in=up] (1,0);
\end{tikzpicture}
\end{aligned}
\,\,=\,\,
0 _{A^\dag \circ A, \id _\FSN}
&
\begin{aligned}
\begin{tikzpicture}
\draw [white] (0,0.2) to (1,0.2);
\draw [reverse arrow=0.05, reverse arrow=0.37, reverse arrow=0.64, reverse arrow=0.96] (0,0)
    to [out=down, in=up] (1,-1.5)
    to [out=down, in=down, looseness=2] (0,-1.5)
    to [out=up, in=down] (1,0);
\end{tikzpicture}
\end{aligned}
\,\,=\,\,
0 _{\id_\FSN, A ^\dag \circ A}
&
\\[-6pt]
& \text{\eqref{eq:biproduct1}} & \text{\eqref{eq:biproduct2}} & \text{\eqref{eq:biproduct3}} &
\end{tabular*}

\vspace{10pt}
\noindent\begin{tabular*}{\textwidth}{@{\extracolsep{\fill}}E*{2}{>{$}c<{$}}E}
&
\begin{aligned}
\begin{tikzpicture}
\draw [white] (0,0.2) to (1,0.2);
\draw [arrow=0.05, arrow=0.5, arrow=0.95] (0,0)
    to [out=up, in=down] (1,1.5) to [out=up, in=down] (0,3);
\draw [reverse arrow=0.05, reverse arrow=0.5, reverse arrow=0.95] (1,0)
    to [out=up, in=down] (0,1.5) to [out=up, in=down] (1,3);
\end{tikzpicture}
\end{aligned}
\,\,=\,\,
\begin{aligned}
\begin{tikzpicture}
\draw [white] (0,0.2) to (1,0.2);
\draw [arrow=0.5] (0,0)to (0,3);
\draw [reverse arrow=0.5] (1,0) to (1,3);
\end{tikzpicture}
\end{aligned}
&
\begin{aligned}
\begin{tikzpicture}
\draw (0,-0.2) to [out=down, in=down, looseness=2] (1,-0.2);
\draw (0,-2.8) to [out=up, in=up, looseness=2] (1,-2.8);
\draw [->] (0,0) to (0,-0.2);
\draw [-<] (1,0) to (1,-0.2);
\draw [-<] (0,-3) to (0,-2.8);
\draw [->] (1,-3) to (1,-2.8);
\end{tikzpicture}
\end{aligned}
\,\,+\,\,
\begin{aligned}
\begin{tikzpicture}
\draw [white] (0,0.2) to (1,0.2);
\draw [reverse arrow=0.05, reverse arrow=0.5, reverse arrow=0.95] (0,0)
    to [out=up, in=down] (1,1.5) to [out=up, in=down] (0,3);
\draw [arrow=0.05, arrow=0.5, arrow=0.95] (1,0)
    to [out=up, in=down] (0,1.5) to [out=up, in=down] (1,3);
\end{tikzpicture}
\end{aligned}
\,\,=\,\,
\begin{aligned}
\begin{tikzpicture}
\draw [white] (0,0.2) to (1,0.2);
\draw [reverse arrow=0.5] (0,0)to (0,3);
\draw [arrow=0.5] (1,0) to (1,3);
\end{tikzpicture}
\end{aligned}
&
\\[4pt]
& \text{\eqref{eq:biproduct4}} & \text{\eqref{eq:biproduct5}} &
\\[5pt]
\end{tabular*}
Khovanov's graphical axioms for the categorified Heisenberg algebra include equations~\eqref{eq:biproduct1}, \eqref{eq:biproduct4} and \eqref{eq:biproduct5}. Equations~\eqref{eq:biproduct2} and \eqref{eq:biproduct3} are not explicitly part of his algebra, but they can be straightforwardly derived from the other three in a setting where hom-set addition distributes over composition and addition is cancellable, properties which hold in both \cat{Span(Gpd)} and also in Khovanov's bimodule category setting.

Although the graphical representations of $i _{A ^\dag \circ A}$ and its converse look like braiding of a braided monoidal category, we do not have such a structure here, as witnessed explicitly by equation~\eqref{eq:biproduct5}: the `braidings' are not invertible. It even fails to be a lax braiding, as naturality fails to hold. We will see in Section~\ref{sec:notnatural} how to deduce this from a theorem of Yetter~\cite{yetter-qg}.

\subsection{Combinatorial interpretation}
\label{sec:interpretation}

We can extend our combinatorial interpretation to the morphisms $i _{\id _\FSN}$, $(i _{\id _\FSN} ) ^\dag$, $i_{A ^\dag \circ A}$ and $(i _{A ^\dag \circ A} ) ^\dag$ as depicted in expression~\eqref{eq:basiccomponents}. This will allow us to see intuitively why equations~(\ref{eq:biproduct1}\_\ref{eq:biproduct5}) should hold.

A span of spans gives a way to relate one history to another, in such a way that related histories have isomorphic sources and isomorphic targets. We can therefore understand our spans of spans by listing the histories which they relate. This does not completely define a span of spans, as the actions on the symmetry groups of objects must also be taken into account, but it is nevertheless a useful way to develop intuition.
\begin{itemize}
\item $\id _\FSN \sxto {i _{\id _\FSN}} A \circ A^\dag$. This relates a history where no change is made to the underlying set, to a history in which some element is added and then removed.
\item $A \circ A ^\dag \sxto {(i _{\id _\FSN} )^\dag} \id _\FSN$. Histories in $A \circ A ^\dag$ representing adding and removing the same element are related to the trivial history representing no action. Histories representing adding one element and removing a different element are related to nothing.
\item $A ^\dag \circ A \sxto {i _{A ^\dag \circ A}} A \circ A ^\dag$. Histories where $x$ is removed and then $y$ is added are related to histories where $y$ is added and then $x$ is removed.
\item $A \circ A ^\dag \sxto{(i _{A ^\dag \circ A})^\dag} A^\dag \circ A$. Given a history where $y$ is added and $x$ is removed, then if $x \neq y$ this is related to the history where $x$ is removed and then $y$ is added; otherwise it is related to nothing.
\end{itemize}
This intuition allows us to understand why equations~(\ref{eq:biproduct1}\_\ref{eq:biproduct5}) should hold. filling out the proof of Lemma~\ref{lem:ccrbiproduct}.

\begin{enumerate}[\hspace{10pt}(1)]
\setcounter{enumi}{25}
\addtocounter{enumi}{-1}
\item If we begin with the trivial action on a set, and then pass to a history where we add and remove the same element, and then verify  that indeed the same element has been added as has been removed, then this leaves our initial history unchanged.
\item Beginning with a history where we remove $x$ and add $y$, we then ensure that $x \neq y$ and pass to a history where we add $y$ and then remove $x$. We then ensure that $x$ and $y$ are the same, which is clearly impossible, so our original history is related to nothing.
\item We begin by choosing a history where we add and remove the same element. We then reverse the order of these operations, which is clearly impossible, so our original history is related to nothing.
\item We begin with a history where we remove $x$ and add $y$. This is related to a history where we add $y$ and then remove $x$, which in turn is related to a history where we remove $x$ and add $y$. This clearly gives the identity on histories.
\item The left-hand side of this equation is comprised of two separate relations on histories. Suppose we begin with a history for which we add $x$ and then remove $y$. The first summand selects the case that $x=y$, and relates it to itself. The second ensures that $x \neq y$ and relates the initial history to the case that we remove $y$ and then add $x$, which in turn is related to the history where we add $x$ and then remove $y$; so this relates every history to itself, except for the case that $x=y$. Overall the sum of these relations give the identity on histories, and so the equation is satisfied.
\end{enumerate}

\subsection{Adjunction of $A$ and $A^{\dagger}$}

The 2-category $\SiiG$ is useful because it contains every object and 1\-cell of $\Gpd$, with the additional property that every 1\-cell has a two-sided adjoint~\cite{spanconstruction}. In particular, for every span, its converse is both a left and a right adjoint. This situation is  called an \textit{ambidextrous adjunction}, or just \textit{ambiadjunction} for short.

In particular, suppose we have a 1\-morphism $\mathbf{A} \sxto{F} \mathbf{B}$ in
$\SiiG$, which is written as the following span:
\begin{equation}
\begin{aligned}
\begin{tikzpicture}[xscale=2, yscale=1.5]
\node (Y) at (0,2) {$\bf X$};
\node (B) at (-1,1) {$\bf B$};
\node (A) at (1,1) {$\bf A$};
\draw [->] (Y) to node [auto, swap] {$G$} (B);
\draw [->] (Y) to node [auto] {$F$} (A);
\end{tikzpicture}
\end{aligned}
\end{equation}
Then its ambiadjoint morphism $\mathbf{B} \sxto{F ^\dag} \mathbf{A}$ is given by the converse span:
\begin{equation}
\begin{aligned}
\begin{tikzpicture}[xscale=2, yscale=1.5]
\node (Y) at (0,2) {$\bf X$};
\node (B) at (-1,1) {$\bf A$};
\node (A) at (1,1) {$\bf B$};
\draw [->] (Y) to node [auto, swap] {$F$} (B);
\draw [->] (Y) to node [auto] {$G$} (A);
\end{tikzpicture}
\end{aligned}
\end{equation}
To fully specify the ambiadjunction we need four 2-morphisms
\begin{align}
\label{eq:etal}
\eta_L & : \id_{A} \Rightarrow F \circ F^{\dagger}
\\ 
\eta_R & : \id_{B} \Rightarrow F^{\dagger} \circ F
\\
\epsilon_L & :  F^{\dagger} \circ F \Rightarrow \id_{B}
\\
\label{eq:epsr}
\epsilon_R & :  F \circ F^{\dagger} \Rightarrow \id_{A}
\end{align}
satisfying the appropriate adjunction equations:
\begin{align}\label{eq:adjproperty}
(\epsilon_R \circ \id_F) \cdot (\id_F \circ \eta_R) &= \id_F
\\
(\id_{F^\dag} \circ \epsilon_R) \cdot (\eta_R \circ \id _{F ^\dag}) &= \id_{F ^\dag}
\\
(\epsilon_L \circ \id_{F^\dag}) \cdot (\id_{F^\dag} \circ \eta_L) &= \id_{F^\dag}
\\
(\id_{F} \circ \epsilon_L) \cdot (\eta_L \circ \id _{F}) &= \id_{F}
\end{align}

We construct these 2-morphisms in the following way. First, $\eta_L$ is formed as follows:
\begin{equation}
\label{eq:etaLdef}
\eta_L \,\,:=
\begin{aligned}
\begin{tikzpicture}[xscale=2, yscale=1.5]
\node (Y) at (0,2) {$\bf A$};
\node (B) at (-1,1) {$\bf A$};
\node (A) at (1,1) {$\bf A$};
\node (X) at (0,1) {$\bf X$};
\node (C) at (0,0) {$(F \downarrow F)$};
\draw [double] (Y) to node [auto, swap] {} (B);
\draw [double] (Y) to node [auto] {} (A);
\draw [->] (X) to node [auto, swap] {$F$} (Y);
\draw [->] (X) to node [auto, swap] {$F$} (B);
\draw [->] (X) to node [auto] {$F$} (A);
\draw [->] (X) to node [auto] {$\Delta_F$} (C);
\draw [->] (C) to node [auto] {$F \circ \pi_1$} (B);
\draw [->] (C) to node [auto, swap] {$F \circ \pi_2$} (A);
\end{tikzpicture}
\end{aligned}
\end{equation}
Here $(G \downarrow G)$ is the iso-comma category given by the weak
pullback of $G$ along itself. Its objects are given by triples
$(x_1,\alpha,x_2)$ such that $\alpha : G(x_1) \ra G(x_2)$, and the morphisms are
compatible pairs of morphisms. Then the diagonal map $\Delta_F$ is the
map into this groupoid from $X$ which on objects maps $x \mapsto
(x,id_{F(x)},x)$ and on morphisms maps $g \mapsto (g,g)$. We form $\eta_R$ in a similar way:
\begin{equation}
\label{eq:unitspan-R}\eta_R \,\,:=
\begin{aligned}
\begin{tikzpicture}[xscale=2, yscale=1.5]
\node (Y) at (0,2) {$\bf B$};
\node (B) at (-1,1) {$\bf B$};
\node (A) at (1,1) {$\bf B$};
\node (X) at (0,1) {$\bf X$};
\node (C) at (0,0) {$(G \downarrow G)$};
\draw [double] (Y) to node [auto, swap] {} (B);
\draw [double] (Y) to node [auto] {} (A);
\draw [->] (X) to node [auto, swap] {$G$} (Y);
\draw [->] (X) to node [auto, swap] {$G$} (B);
\draw [->] (X) to node [auto] {$G$} (A);
\draw [->] (X) to node [auto] {$\Delta_G$} (C);
\draw [->] (C) to node [auto] {$G \circ \pi_1$} (B);
\draw [->] (C) to node [auto, swap] {$G \circ \pi_2$} (A);
\end{tikzpicture}
\end{aligned}
\end{equation}
The counit 2-cells are then formed as the converses of these:
\begin{align}
\label{eq:counitspan-L}
  \epsilon_L &:= \eta_R {}^{\ast}
\\
\label{eq:epsRdef}
  \epsilon_R &:= \eta_L {}^{\ast}
\end{align}

\begin{lemma}
  The 2-cells (\ref{eq:etaLdef})-(\ref{eq:epsRdef}) are the
  units and counits of an ambidextrous adjunction.
\end{lemma}
\begin{proof}

  It is obvious that these 2-morphisms have the correct source and
  target. To see that they satisfy (\ref{eq:adjproperty}), there are
  four properties to check (two for each adjunction), but the proofs
  are all essentially the same.  Consider the identity
  \begin{equation}
    ( \id \circ \eta_R ) \cdot ( \epsilon_R \circ \id ) = \id 
  \end{equation}

The left hand side is given by the following diagram, representing the composite of 2-morphisms:
\begin{equation}
\label{eq:adjunction-composite}
\begin{aligned}
\begin{tikzpicture}[yscale=-1, scale=1.5]
\node (1) at (0,0) {$\bf B$};
\node (2) at (1,0) {$\bf X$};
\node (3) at (2,0) {$\bf A$};
\node (4) at (4,0) {$\bf A$};
\node (5) at (6,0) {$\bf A$};
\node (6) at (0,2) {$\bf B$};
\node (7) at (1,2) {$\bf X$};
\node (8) at (2,2) {$\bf A$};
\node (9) at (4,1) {$\bf X$};
\node (10) at (4,2) {$(G \downarrow G)$};
\node (11) at (6,2) {$\bf B$};
\node (12) at (0,3) {$\bf B$};
\node (13) at (1,3) {$\bf X$};
\node (14) at (2,3) {$\bf A$};
\node (15) at (3,3) {$\bf X$};
\node (16) at (4,3) {$\bf B$};
\node (17) at (5,3) {$\bf X$};
\node (18) at (6,3) {$\bf A$};
\node (19) at (0,4) {$\bf B$};
\node (20) at (2,4) {$(F \downarrow F)$};
\node (21) at (4,4) {$\bf B$};
\node (22) at (5,4) {$\bf X$};
\node (23) at (6,4) {$\bf A$};
\node (24) at (2,5) {$\bf X$};
\node (25) at (0,6) {$\bf B$};
\node (26) at (2,6) {$\bf B$};
\node (27) at (4,6) {$\bf B$};
\node (28) at (5,6) {$\bf X$};
\node (29) at (6,6) {$\bf A$};
\draw [->] (2) to node [auto, swap] {$G$} (1);
\draw [->] (2) to node [auto] {$F$} (3);
\draw [->] (9) to node [auto, swap] {$F$} (4);
\draw [->] (9) to node [auto] {$\Delta_G$} (10);
\draw [->] (7) to node [auto, swap] {$G$} (6);
\draw [->] (7) to node [auto] {$F$} (8);
\draw [->] (10) to node [auto, swap] {$F \circ \pi_1$} (8);
\draw [->] (10) to node [auto] {$G \circ \pi_2$} (11);
\draw [->] (10) to node [auto, swap] {$\pi_1$} (15);
\draw [->] (10) to node [auto] {$\pi_2$} (17);
\draw [->] (13) to node [auto, swap] {$G$} (12);
\draw [->] (13) to node [auto] {$F$} (14);
\draw [->] (15) to node [auto, swap] {$F$} (14);
\draw [->] (15) to node [auto] {$G$} (16);
\draw [->] (17) to node [auto, swap] {$G$} (16);
\draw [->] (17) to node [auto] {$F$} (18);
\draw [->] (20) to node [auto] {$\pi_1$} (13);
\draw [->] (20) to node [auto, swap] {$\pi_2$} (15);
\draw [->] (20) to node [auto] {$G \circ \pi_1$} (19);
\draw [->] (20) to node [auto, swap] {$F \circ \pi_2$} (21);
\draw [->] (22) to node [auto, swap] {$G$} (21);
\draw [->] (22) to node [auto] {$F$} (23);
\draw [->] (24) to node [auto] {$\Delta_F$} (20);
\draw [->] (24) to node [auto] {$G$} (26);
\draw [->] (28) to node [auto, swap] {$G$} (27);
\draw [->] (28) to node [auto] {$F$} (29);
\draw [double equal sign distance] (3) to (4);
\draw [double equal sign distance] (4) to (5);
\draw [double equal sign distance] (1) to (6);
\draw [double equal sign distance] (2) to (7);
\draw [double equal sign distance] (3) to (8);
\draw [double equal sign distance] (5) to (11);
\draw [double equal sign distance] (6) to (12);
\draw [double equal sign distance] (7) to (13);
\draw [double equal sign distance] (8) to (14);
\draw [double equal sign distance] (11) to (18);
\draw [double equal sign distance] (12) to (19);
\draw [double equal sign distance] (16) to (21);
\draw [double equal sign distance] (17) to (22);
\draw [double equal sign distance] (18) to (23);
\draw [double equal sign distance] (19) to (25);
\draw [double equal sign distance] (25) to (26);
\draw [double equal sign distance] (26) to (27);
\draw [double equal sign distance] (21) to (27);
\draw [double equal sign distance] (22) to (28);
\draw [double equal sign distance] (23) to (29);
\node at ($ (14.center) !.42! (20.center) $) {$\stackrel \beta \Rightarrow$};
\node at ($ (10.center) !.5! (16.center) $) {$\stackrel \alpha \Leftarrow$};
\end{tikzpicture}
\end{aligned}
\end{equation}
Note that $\eta_R$ appears in the top right, and $\epsilon_R$ in the
  bottom left of this diagram. The middle rows simply relate the
  composite $F \circ F^{\dagger} \circ F$ to the composites in the
  unit and counit by exhibiting the weak pullback which gives the
  composite.

  To see that the composite of these 2-cells is just the identity, we
  simply use the fact that the diagonal map has a special role
  relative to the weak pullback. Namely, in the following diagram we have that $\pi_1 \circ \Delta_G = \pi_2 \circ \Delta_G$:
\begin{equation}
\begin{aligned}
\begin{tikzpicture}[xscale=2, yscale=1.5]
\node (1) at (0,0) {$\bf X$};
\node (2) at (1,0) {$(G \downarrow G)$};
\node (3) at (2,1) {$\bf X$};
\node (4) at (2,0) {$\bf B$};
\node (5) at (2,-1) {$\bf X$};
\draw [->] (1) to node [auto] {$\Delta_G$} (2);
\draw [->] (2) to node [auto] {$\pi_1$} (3);
\draw [->] (2) to node [auto, swap] {$\pi_2$} (5);
\draw [->] (3) to node [auto] {$G$} (4);
\draw [->] (5) to node [auto, swap] {$G$} (4);
\node at (1.6,0) {$\Downarrow \!\! \alpha$};
\end{tikzpicture}
\end{aligned}
\end{equation}
 Furthermore, by construction of $\Delta_G$, the composite $\alpha \circ \Delta_G$ gives the identity natural transformation. So this 2-morphism is just the same as the identity:
\begin{equation}
\begin{aligned}
\begin{tikzpicture}[xscale=2, yscale=1.5]
\node (1) at (0,0) {$\bf X$};
\node (3) at (2,1) {$\bf X$};
\node (4) at (2,0) {$\bf B$};
\node (5) at (2,-1) {$\bf X$};
\draw [double equal sign distance] (1) to node [auto] {$$} (3);
\draw [double equal sign distance] (1) to node [auto, swap] {$$} (5);
\draw [->] (3) to node [auto] {$G$} (4);
\draw [->] (5) to node [auto, swap] {$G$} (4);
\end{tikzpicture}
\end{aligned}
\end{equation}
Applying this at both the unit and counit in~(\ref{eq:adjunction-composite}) one readily verifies that the whole composite is just the identity. The other three identities for an ambiadjunction are proved in precisely the same way, and hence our given 2-morphisms are indeed unit and counit cells for an ambiadjunction.
\end{proof}

In particular, we have the immediate special case:

\begin{corollary}\label{lemma:ambiadj}
There is an ambiadjunction between the spans $A$ and $A^{\dag}$.
\end{corollary}

\noindent
Thus we have, for example, the unit 2-morphism:
\begin{equation}\label{eq:unit-twocell}
\begin{aligned}
\begin{tikzpicture}[xscale=2, yscale=1.5]
\node (X) at (0,0) {$\FSN$};
\node (Y) at (0,2) {$\FSN$};
\node (Z) at (0,1) {$\FSN$};
\node (B) at (-1,1) {$\FSN$};
\node (A) at (1,1) {$\FSN$};
\draw [->] (X) to node [auto] {$+1$} (B);
\draw [->] (X) to node [auto, swap] {$+1$} (A);
\draw [->] (Y) to node [auto] {$\id_\FSN$} (A);
\draw [->] (Y) to node [auto, swap] {$\id_\FSN$} (B);
\draw [->] (Z) to node [auto] {$\id _\FSN$} (X);
\draw [->] (Z) to node [auto, swap] {$+1$} (Y);
\draw [<-, double] (0.4,0.7) to
    node [right, pos=0.5] {$\id$}
    (0.4,1.3);
\draw [<-, double] (-0.4,0.7) to
    node [right, pos=0.5] {$\id$}
    (-0.4,1.3);
\node at (0,-0.5) {$\id_\FSN \sxto{\eta} A ^\dag \circ A$};
\end{tikzpicture}
\end{aligned}
\end{equation}
The graphical representation is extended to depict $\eta$ and its converse in the following way:
\begin{align}
\centerdia{\begin{aligned}
\begin{tikzpicture}
\draw [white] (0,0) to (0,-1.0);
\draw (0,0) to (0,-0.2) to [out=down, in=down, looseness=2] (1,-0.2) to (1,0);
\draw [>-] (0,-0.2) to (0,0);
\draw [->] (1,0) to (1,-0.2);
\end{tikzpicture}
\end{aligned}}
&&
\centerdia{\begin{aligned}
\begin{tikzpicture}
\draw [white] (0,-0.2) to (0,0.8);
\draw (0,-0.2) to (0,0) to [out=up, in=up, looseness=2] (1,0) to (1,-0.2);
\draw [->] (0,-0.2) to (0,0);
\draw [>-] (1,0) to (1,-0.2);
\end{tikzpicture}
\end{aligned}}
\\
\centerdia{\eta}
&&
\centerdia{\eta ^\dag}
\end{align}
The equations for the ambiadjunction then have the following
representation in our graphical notation:
\begin{equation}
\label{eq:snakeequations}
\begin{tabular}{c@{\hspace{80pt}}c}
$\begin{aligned}
\begin{tikzpicture}[scale=0.9]
\draw [arrow=0.05, arrow=0.5, arrow=0.95] (0,0) to (0,1.5) to [out=up, in=up, looseness=2] (1,1.5) to [out=down, in=down, looseness=2] (2,1.5) to (2,3);
\end{tikzpicture}
\end{aligned}
\quad=\quad
\begin{aligned}
\begin{tikzpicture}[scale=0.9]
\draw [arrow=0.5] (0,0) to (0,3);
\end{tikzpicture}
\end{aligned}$
&
$\begin{aligned}
\begin{tikzpicture}[scale=0.9]
\draw [reverse arrow=0.05, reverse arrow=0.5, reverse arrow=0.95] (0,0) to (0,1.5) to [out=up, in=up, looseness=2] (-1,1.5) to [out=down, in=down, looseness=2] (-2,1.5) to (-2,3);
\end{tikzpicture}
\end{aligned}
\quad=\quad
\begin{aligned}
\begin{tikzpicture}[scale=0.9]
\draw [reverse arrow=0.5] (0,0) to (0,3);
\end{tikzpicture}
\end{aligned}$
\\\vspace{5pt}\\
$\begin{aligned}
\begin{tikzpicture}[scale=0.9]
\draw [reverse arrow=0.05, reverse arrow=0.5, reverse arrow=0.95] (0,0) to (0,1.5) to [out=up, in=up, looseness=2] (1,1.5) to [out=down, in=down, looseness=2] (2,1.5) to (2,3);
\end{tikzpicture}
\end{aligned}
\quad=\quad
\begin{aligned}
\begin{tikzpicture}[scale=0.9]
\draw [reverse arrow=0.5] (0,0) to (0,3);
\end{tikzpicture}
\end{aligned}$
&
$\begin{aligned}
\begin{tikzpicture}[scale=0.9]
\draw [arrow=0.05, arrow=0.5, arrow=0.95] (0,0) to (0,1.5) to [out=up, in=up, looseness=2] (-1,1.5) to [out=down, in=down, looseness=2] (-2,1.5) to (-2,3);
\end{tikzpicture}
\end{aligned}
\quad=\quad
\begin{aligned}
\begin{tikzpicture}[scale=0.9]
\draw [arrow=0.5] (0,0) to (0,3);
\end{tikzpicture}
\end{aligned}$
\end{tabular}
\end{equation}
One can readily see that these diagrams show exactly the ``path'' of a chain of $\id_X$ arrows in the big composite (\ref{eq:adjunction-composite}).

We can understand how $\eta$ acts with the following diagram which depicting the action of $\eta$ at the 2-element set; a similar picture could be drawn for each nonempty set in \FSN.
\begin{equation}
\begin{aligned}
\includegraphics{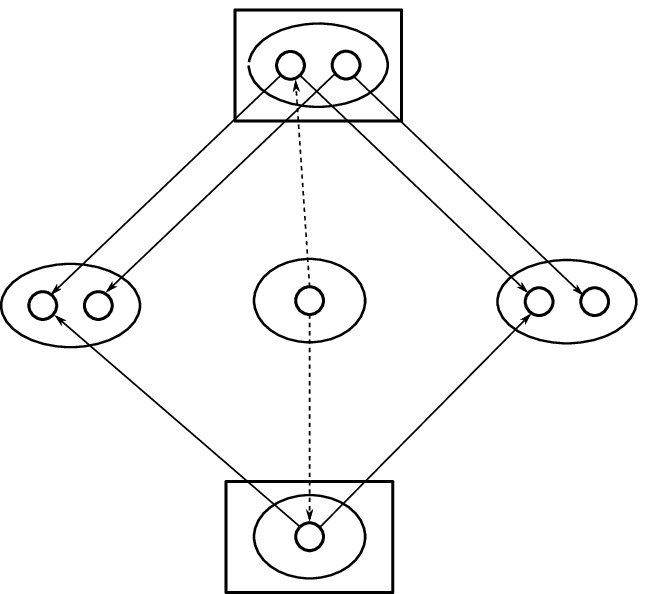}
\end{aligned}
\end{equation}
This can be compared to (\ref{diag:i-id-s}), which similarly illustrates $i _{\id _\FSN}$. The converse spans of spans are just the same, with the top-to-bottom orientation reversed.

This picture is a helpful aid to extending our combinatorial interpretation for other structures, given in Section~\ref{sec:interpretation}, to include $\eta$ and the snake equations (\ref{eq:snakeequations}). The span of spans $\eta$, and its converse, are interpreted in the following way. 
\begin{itemize}
\item $\id _\FSN \sxto{\eta} A ^\dag \circ A$. This relates the identity history to the history where an element is removed, and then added. This is impossible on the zero-element set, and so in that case $\eta$ relates the identity history to nothing.
\item $A ^\dag \circ A \sxto{\eta ^\dag} \id _\FSN$. This relates any history to the identity history.
\end{itemize}
The equations (\ref{eq:snakeequations}) can be established combinatorially in a similar way to equations~(\ref{eq:biproduct1}\_\ref{eq:biproduct5}). We examine the first  in detail. One the left-hand side, we begin with the operator $A ^\dag$, a history of which corresponds to adding some element $x$ to a set. We then apply $i _{\id _\FSN}$, passing to a history for which we add some element $y$, remove it, and then add our element $x$. Finally we apply our interpretation of $\eta ^\dag$, giving a history where we simply add the element $y$. This is not the same as our original history, for which we added the element $x$, but it is \textit{equivalent} to it, which is our definition of equality for 2-morphisms in \cat{Span(Gpd)} as given by equation~\eqref{eq:equivalencerelation}. The other snake equations can be interpreted in a similar way.

\subsection{Symmetric group actions}
\label{sec:symgrpact}

The construction of the full categorified Heisenberg algebra makes use of certain symmetrizer objects. These are subobjects invariant under the action of the symmetric groups $S_n$ on spans of the form $A ^n$ or $(A ^\dag) ^n$. Roughly, the effect of this action is to permute the order in which elements of a set are added or removed. We now describe it explicitly.
\begin{lemma}\label{lemma:symgrpaction}
There exist actions of $S_n$ on $A^n$ and $(A^{\dag})^n$.
\end{lemma}
\begin{proof}
Up to isomorphism, spans of the form $A ^n$ and $(A ^\dag) ^n$ have the following form:
\begin{equation}
\begin{aligned}
\begin{tikzpicture}[xscale=2, yscale=1.5]
\node (T) at (0,0) {\FSN};
\node (L) at (-1,-1) {\FSN};
\node (R) at (1,-1) {\FSN};
\draw [->] (T) to node [auto, swap] {$\id _\FSN$} (L);
\draw [->] (T) to node [auto] {$+n$} (R);
\end{tikzpicture}
\end{aligned}
\hspace{70pt}
\begin{aligned}
\begin{tikzpicture}[xscale=2, yscale=1.5]
\node (T) at (0,0) {\FSN};
\node (L) at (-1,-1) {\FSN};
\node (R) at (1,-1) {\FSN};
\draw [->] (T) to node [auto, swap] {$+n$} (L);
\draw [->] (T) to node [auto] {$\id_\FSN$} (R);
\end{tikzpicture}
\end{aligned}
\end{equation}
Here we define $(+n):=(+1)^n$, the functor which takes the disjoint union with an $n$\-element set. It has natural endo-transformations $\hat{\mu}$ for each $\mu \in \Sn$, defined for each set $T$ by the set maps
\begin{equation}\label{eq:symnattrans}
  \hat{\mu}_T : T \sqcup \mathbf{n} \ra T \sqcup \mathbf{n},
\end{equation}
which act by $\hat{\mu}_T(t) = t$ for all $t \in T$ and
$\hat{\mu}_T(j) = \mu(j)$ for all $j \in \mathbf{n}$. These give rise to the following spans of spans:
\begin{align}
\begin{aligned}
\begin{tikzpicture}[xscale=2, yscale=1.5]
\node (X) at (0,0) {$\FSN$};
\node (Y) at (0,-2) {$\FSN$};
\node (Z) at (0,-1) {$\FSN$};
\node (B) at (-1,-1) {$\FSN$};
\node (A) at (1,-1) {$\FSN$};
\draw [->] (X) to node [auto, swap] {$\id_\FSN$} (B);
\draw [->] (X) to node [auto] {$+n$} (A);
\draw [->] (Y) to node [auto, swap] {$+n$} (A);
\draw [->] (Y) to node [auto] {$\id_\FSN$} (B);
\draw [->] (Z) to node [auto, swap] {$\id_\FSN$} (X);
\draw [->] (Z) to node [auto] {$\id_\FSN$} (Y);
\draw [->, double] (0.4,-0.7) to
    node [right, pos=0.5] {$\hat \mu$}
    (0.4,-1.3);
\draw [->, double] (-0.4,-0.7) to
    node [right, pos=0.5] {$\id$}
    (-0.4,-1.3);
\end{tikzpicture}
\end{aligned}
&&
\begin{aligned}
\begin{tikzpicture}[xscale=2, yscale=1.5]
\node (X) at (0,0) {$\FSN$};
\node (Y) at (0,-2) {$\FSN$};
\node (Z) at (0,-1) {$\FSN$};
\node (B) at (-1,-1) {$\FSN$};
\node (A) at (1,-1) {$\FSN$};
\draw [->] (X) to node [auto, swap] {$+n$} (B);
\draw [->] (X) to node [auto] {$\id_\FSN$} (A);
\draw [->] (Y) to node [auto, swap] {$\id_\FSN$} (A);
\draw [->] (Y) to node [auto] {$+n$} (B);
\draw [->] (Z) to node [auto, swap] {$\id_\FSN$} (X);
\draw [->] (Z) to node [auto] {$\id_\FSN$} (Y);
\draw [->, double] (0.4,-0.7) to
    node [right, pos=0.5] {$\id$}
    (0.4,-1.3);
\draw [->, double] (-0.4,-0.7) to
    node [right, pos=0.5] {$\hat \mu$}
    (-0.4,-1.3);
\end{tikzpicture}
\end{aligned}
\end{align}
Intuitively these permute the extra $n$ elements, and leave the remaining elements unchanged.
\end{proof}

For the action of the non-identity element of $S_2$ on $A^2$ and $(A ^\dag) ^2$, we use the following graphical representation:
\begin{align}
\label{eq:symmetrygraphically}
\centerdia{\begin{aligned}
\begin{tikzpicture}
\draw [reverse arrow=0.1, reverse arrow=0.9] (0,0) to [out=up, in=down] (1,1.5);
\draw [reverse arrow=0.1, reverse arrow=0.9] (1,0) to [out=up, in=down] (0,1.5);
\end{tikzpicture}
\end{aligned}}
&&
\centerdia{\begin{aligned}
\begin{tikzpicture}
\draw [arrow=0.1, arrow=0.9] (0,0) to [out=up, in=down] (1,1.5);
\draw [arrow=0.1, arrow=0.9] (1,0) to [out=up, in=down] (0,1.5);
\end{tikzpicture}
\end{aligned}}
\end{align}
Actions of $S_n$ on $A^n$ and $(A ^\dag) ^n$ can be built up from these basic permutations of adjacent operations. Because they are actions of the symmetric group, they satisfy the following equations:
\begin{align}
\begin{aligned}
\begin{tikzpicture}
\draw [white] (0,0.2) to (1,0.2);
\draw [reverse arrow=0.05, reverse arrow=0.5, reverse arrow=0.95] (0,0)
    to [out=up, in=down] (1,1.5) to [out=up, in=down] (0,3);
\draw [reverse arrow=0.05, reverse arrow=0.5, reverse arrow=0.95] (1,0)
    to [out=up, in=down] (0,1.5) to [out=up, in=down] (1,3);
\end{tikzpicture}
\end{aligned}
\quad=\quad
\begin{aligned}
\begin{tikzpicture}
\draw [white] (0,0.2) to (1,0.2);
\draw [reverse arrow=0.5] (0,0)to (0,3);
\draw [reverse arrow=0.5] (1,0) to (1,3);
\end{tikzpicture}
\end{aligned}
&&
\begin{aligned}
\begin{tikzpicture}
\draw [white] (0,0.2) to (1,0.2);
\draw [arrow=0.05, arrow=0.5, arrow=0.95] (0,0)
    to [out=up, in=down] (1,1.5) to [out=up, in=down] (0,3);
\draw [arrow=0.05, arrow=0.5, arrow=0.95] (1,0)
    to [out=up, in=down] (0,1.5) to [out=up, in=down] (1,3);
\end{tikzpicture}
\end{aligned}
\quad=\quad
\begin{aligned}
\begin{tikzpicture}
\draw [white] (0,0.2) to (1,0.2);
\draw [arrow=0.5] (0,0)to (0,3);
\draw [arrow=0.5] (1,0) to (1,3);
\end{tikzpicture}
\end{aligned}
\end{align}
They also satisfy the following equations, making them generators of $S_n$ for any $n$:
\begin{align}
\begin{aligned}
\begin{tikzpicture}
\draw [arrow=0.08, arrow=0.92] (0,0)
    to [out=up, in=down] (1,1)
    to [out=up, in=down] (0,2);
\draw [arrow=0.08, arrow=0.92] (-1,0) to [out=up, in=down] (1,2);
\draw [arrow=0.08, arrow=0.92] (1,0) to [out=up, in=down] (-1,2);
\end{tikzpicture}
\end{aligned}
\quad=\quad
\begin{aligned}
\begin{tikzpicture}
\draw [arrow=0.08, arrow=0.92] (0,0)
    to [out=up, in=down] (-1,1)
    to [out=up, in=down] (0,2);
\draw [arrow=0.08, arrow=0.92] (-1,0) to [out=up, in=down] (1,2);
\draw [arrow=0.08, arrow=0.92] (1,0) to [out=up, in=down] (-1,2);
\end{tikzpicture}
\end{aligned}
&&
\begin{aligned}
\begin{tikzpicture}
\draw [reverse arrow=0.08, reverse arrow=0.92] (0,0)
    to [out=up, in=down] (1,1)
    to [out=up, in=down] (0,2);
\draw [reverse arrow=0.08, reverse arrow=0.92] (-1,0) to [out=up, in=down] (1,2);
\draw [reverse arrow=0.08, reverse arrow=0.92] (1,0) to [out=up, in=down] (-1,2);
\end{tikzpicture}
\end{aligned}
\quad=\quad
\begin{aligned}
\begin{tikzpicture}
\draw [reverse arrow=0.08, reverse arrow=0.92] (0,0)
    to [out=up, in=down] (-1,1)
    to [out=up, in=down] (0,2);
\draw [reverse arrow=0.08, reverse arrow=0.92] (-1,0) to [out=up, in=down] (1,2);
\draw [reverse arrow=0.08,reverse arrow=0.92] (1,0) to [out=up, in=down] (-1,2);
\end{tikzpicture}
\end{aligned}
\end{align}
As relations on histories, the symmetry on $A \circ A$ relates the history `remove $x$, and then remove $y$' to the history `remove $y$, and then remove $x$'. The symmetry on $A ^\dag \circ A ^\dag$ can be described in a similar way. The equations satisfied by these symmetries can then be accounted for using these combinatorial interpretations.

Together with the morphism $i _{A ^\dag \circ A}$ and its converse, these morphisms allow us to interpret arbitrary braid diagrams involving strands labeled $A$ or $A ^\dag$. These are not invertible in general, since $\injAsA$ is not an isomorphism but an inclusion, and so we do not get a representation of the symmetric group in these cases where both $A$ and $A^\dag$ appear.

This is important for connections to quantum field theory, for which powers of the field operator $\Phi = A + A^{\dagger}$ play a major role.  The inner product
$\inprod{\psi,p(\phi)\psi'}$ (for a polynomial $p$) is the groupoid
cardinality of the stuff type inner product
\begin{equation}
\inprod{\Psi,p(\Phi)\Psi'}
\end{equation}
where $\Psi$, $\Psi '$ are any groupoidifications of the vectors
$\psi$, $\psi '$.  As described in \cite{mstuff}, this can be
interpreted as a \textit{sum over histories}, which are represented by
Feynman diagrams.  These are given weights which are exactly as
determined by the groupoid cardinality.

\subsection{Failure of naturality}
\label{sec:notnatural}

The crossings in our graphical notation are self-inverse when all the strands are oriented in the same direction. This is the basis of the symmetric group action on powers such as $A^n$ and $(A^{\dag})^n$. However, this is not the case for `mixed' crossings in which the strands are oriented in opposite directions, a fact implied by \eqref{eq:biproduct5}. Indeed, this is essential to the categorification of the commutation relations.

By a theorem of Yetter~\cite{yetter-qg}, a lax braiding for an object with a right dual is automatically invertible. Since our `braidings' are not invertible, one of the axioms of a lax braiding must fail, and the culprit is naturality. If we \textit{assume} naturality, we can use Yetter's argument to show that the braiding is invertible, in contradiction with equation~\eqref{eq:biproduct5}:
\allowdisplaybreaks[1]
\begin{align}
\nonumber
&\begin{aligned}
\begin{tikzpicture}
\draw [reverse arrow={0.05}, reverse arrow=0.95] (0,-2) to (0,0)
    to [out=up, in=down] (1,1)
    to [out=up, in=down] (0,2);
\draw [arrow={0.05}, arrow=0.95] (1,-2) to (1,0)
    to [out=up, in=down] (0,1)
    to [out=up, in=down] (1,2);
\end{tikzpicture}
\end{aligned}
\quad=\quad
\begin{aligned}
\begin{tikzpicture}
\draw [arrow=0.03, arrow=0.97] (1,0)
    to (1,1)
    to [out=up, in=up, looseness=1.5] (0,1)
    to [out=down, in=down, looseness=1.5] (-1,1)
    to [out=up, in=down] (-1,2)
    to [out=up, in=down] (-2,3)
    to [out=up, in=down] (-1,4);
\draw [reverse arrow=0.05, reverse arrow=0.95] (-2,0)
    to [out=up, in=down] (-2,1)
    to [out=up, in=down] (-2,2)
    to [out=up, in=down] (-1,3)
    to [out=up, in=down] (-2,4);
\end{tikzpicture}
\end{aligned}
\quad=\quad
\begin{aligned}
\begin{tikzpicture}
\draw [arrow=0.027, arrow=0.975] (1,0)
    to (1,1)
    to [out=up, in=up] (-1,1)
    to [out=down, in=down, looseness=1.5] (-2,1)
    to [out=up, in=down] (-1,2)
    to [out=up, in=down] (-2,3)
    to [out=up, in=down] (-1,4);
\draw [reverse arrow=0.037, reverse arrow=0.957] (0,0)
    to [out=up, in=down] (0,1)
    to [out=up, in=down] (-2,2)
    to [out=up, in=down] (-1,3)
    to [out=up, in=down] (-2,4);
\end{tikzpicture}
\end{aligned}
\\
&\hspace{20pt}
\quad=\quad
\begin{aligned}
\begin{tikzpicture}
\draw [arrow=0.027, arrow=0.97] (1,0)
    to (1,1)
    to [out=up, in=up] (-1,1)
    to [out=down, in=down, looseness=1.5] (-2,1)
    to [out=up, in=down] (-2,3)
    to [out=up, in=down] (-1,4);
\draw [reverse arrow=0.045, reverse arrow=0.95] (0,0)
    to [out=up, in=down] (0,1)
    to [out=up, in=down] (-1,2)
    to [out=up, in=down] (-1,3)
    to [out=up, in=down] (-2,4);
\end{tikzpicture}
\end{aligned}
\quad=\quad
\begin{aligned}
\begin{tikzpicture}
\draw [arrow=0.035, arrow=0.966] (1,0)
    to (1,2)
    to [out=up, in=up, looseness=1.5] (0,2)
    to [out=down, in=down, looseness=1.5] (-1,2)
    to [out=up, in=down] (-1,3)
    to [out=up, in=down] (-1,4);
\draw [reverse arrow=0.045, reverse arrow=0.95] (0,0)
    to [out=up, in=down] (-2,2)
    to [out=up, in=down] (-2,4);
\end{tikzpicture}
\end{aligned}
\quad=\quad
\begin{aligned}
\begin{tikzpicture}
\draw [arrow=0.5] (1,0)
    to [out=up, in=down] (1,4);
\draw [reverse arrow=0.5] (0,0)
    to [out=up, in=down] (0,4);
\end{tikzpicture}
\end{aligned}
\end{align}
Combinatorially, the following naturality property fails to hold.
\begin{lemma}
The crossings~\eqref{eq:basiccomponents} violate naturality in the combinatorial representation.
\end{lemma}
\begin{proof}
The following composites are not equal:
\begin{equation}
\begin{aligned}
\begin{tikzpicture}
\draw [arrow=0.045, arrow=0.97] (0,3)
    to [out=down, in=up] (-1,1)
    to [out=down, in=down, looseness=1.5] (-2,1)
    to [out=up, in=down] (-1,3);
\draw [reverse arrow=0.045, reverse arrow=0.95] (0,0)
    to [out=up, in=down] (-2,3);
\end{tikzpicture}
\end{aligned}
\quad\ne\quad
\begin{aligned}
\begin{tikzpicture}
\draw [arrow=0.075, arrow=0.95] (0,3)
    to [out=down, in=up] (0,2)
    to [out=down, in=down, looseness=1.5] (-1,2)
    to [out=up, in=down] (-1,3);
\draw [reverse arrow=0.06, reverse arrow=0.94] (-2,0)
    to [out=up, in=down] (-2,3);
\end{tikzpicture}
\end{aligned}
\end{equation}
The reason this equality fails is that the left-hand side factors through the composite $A^\dag \circ A \circ A$, and hence annihilates the history that removes the unique element from a 1\-element set. The right-hand does not annihilate this history, and hence the two spans of spans cannot be equal.
\end{proof}

\subsection{The twist}\label{sec:twist}

\begin{lemma}
\label{lem:twist}
Khovanov's `left twist equals zero' axiom holds in the combinatorial representation:
\begin{equation}
\begin{aligned}
\begin{tikzpicture}
\draw [arrow=0.05, arrow=0.5, arrow=0.95] (0,0)
    to (0,1)
    to [out=up, in=down] (-1,2)
    to [out=up, in=up, looseness=1.5] (-2,2)
    to (-2,1)
    to [out=down, in=down, looseness=1.5] (-1,1)
    to [out=up, in=down] (0,2)
    to (0,3);
\end{tikzpicture}
\end{aligned}
\quad=\quad
0 _{A ^\dag, A ^\dag}
\end{equation}
\end{lemma}
\begin{proof}
We verify this combinatorially as follows. The initial span is $A ^\dag$, and we begin with a history in this span which adds an element $x$ to a set. This is related to a history in which we add $x$, then add and remove some new element $y$, according to our interpretation of $i _{\id _\FSN}$ as described above. We then apply the symmetry map, obtaining a history in which we add $y$, add $x$, and then remove $y$. Finally we apply our interpretation of $(i _{\id _\FSN}) ^\dag$ to the final add-remove pair, which only relates histories in which the element we remove is the same as the element which we add. However, in this case this is impossible, as $x$ and $y$ are different elements, and so our history is related to nothing. Combinatorially, this just says that we are \textit{distinguishing} $x$ and $y$.
\end{proof}

\subsection{Main theorem}\label{sec:rep2catthm}

\newcommand\HZ{\ensuremath{\cat{H' _{\mathbb{Z}_+}}}}
To this point, we have described a combinatorial interpretation of the
graphical notation for the monoidal category $\cat H'$. This
interpretation may be understood as a representation as living in the 2-category
$\cat{Span(Gpd)}$, as  a functor from \cat{H'} to hom-category of spans on the object $\FSN$. This interpretation is summarized in Table \ref{table:correspondence}. One can equivalently build a 2\-category $\cat{\Omega_{H'}}$ with one object, whose morphisms and 2-morphisms are the objects and morphisms of \cat{H'}. Then our representation is exactly a 2\-functor
$$\cat{\Omega_{H'}} \sxto C \cat{Span(Gpd)}.$$
\begin{table}
\begin{center}
  \begin{tabular}{|l|l|}
  \hline
    \bf Khovanov's $\cat{H'}$ & \bf Image in $\SiiG$ \\
    \hline
    (none) & $\FSN$ \\
    $Q_-$, $Q_+$ & $A$, $A^{\dagger}$ \\
    $\downarrow$, $\uparrow$ & $\id_A$, $\id_{A^{\dagger}}$ \\
    $\otimes$ & $\circ$ \\
    $\cap$, $\cup$ & $\eta$, $\epsilon$ \\
    crossings & $\hat{\mu}$, $\injAsA$, $(\injAsA)^{\ast}$ \\
  \hline
  \end{tabular}
\end{center}
\caption{\label{table:correspondence}The correspondence between categorifications} 
\end{table}This is formalized by the following theorem.
\begin{theorem}\label{thm:mainthm}
  The combinatorial interpretation of Khovanov's categorified Heisenberg algebra gives a representation of \cat K on the object $\FSN$ in the 2-category
  \cat{Span(Gpd)}.
\end{theorem}
\begin{proof}
  Most of the necessary facts have already been proved. The functor
  takes the down- and up-arrows in the graphical notation for
  $\cat{H'}$ to the spans $A$ and $A^{\dag}$ respectively, which
  determines all products. The cups, caps, and crossings are taken to
  those spans of spans which we have previously described. These also
  determine all monoidal products of these 2-morphisms.

  Since the functor is defined by its action on a generating set of
  objects and morphisms in~$\cat {H'}$, by construction it will
  automatically respect composition, identities and the monoidal
  structure, provided it is well-defined. In particular, one must
  verify that the spans of spans we have defined satisfy the relations
  imposed on the graphical notation for $\cat{H'}$, since otherwise
  the definition might give different spans of spans for equivalent
  2-morphisms in~$\cat{H'}$.

  The adjointness relations on $i_{\id_{\FSN}}$ and
  $(i_{\id_{\FSN}})^{\dag}$ were verified in Lemma
  \ref{lemma:ambiadj}. The relations on crossings for like-oriented
  strands are established by the existence of the symmetric group
  action shown in Lemma \ref{lemma:symgrpaction}. The crossing
  relations for strands with unlike orientations are exactly the
  biproduct equations (\ref{eq:biproduct4}) and (\ref{eq:biproduct5}).

  The remaining relations are the `left twist equals zero' relation
  and the `loop cancellation' relation. The first of these is the only
  nontrivial one, which was verified in Lemma~\ref{lem:twist}.
\end{proof}
 
\section{Functorial representations}
\label{sec:linear}

\subsection{Introduction}

Until now we have been describing a representation of the
categorified Heisenberg algebra \cat {H'} on the groupoid $\FSN$ of finite sets as an object in $\SiiG$.
It is more usual to look for representations of categorified
algebras in $\cat{Cat}$, the 2-category of categories, functors, and
natural transformations \cite{mazorchuk}. Often, one asks further that
the categories be equipped with extra structure, as with abelian,
additive, or triangulated categories, and functors are required
preserve this structure.

In this section, we show two ways to get representations by
functors. First, we take a detour through a somewhat different
representation by functors on a category, which relates to earlier
work on groupoidification.

The key point is that the combinatorial representation in $\SiiG$ acts as the `combinatorial core' of these functorial representation. This is precisely in the spirit of the
groupoidification program which inspired the structures in
$\SiiG$. The bulk of this Section will explain these, and in
particular will demonstrate how to recover one which is equivalent to
the representation described by Khovanov in terms of bimodules~\cite{khovanov}, which also includes certain `symmetrizer'
2-morphisms.

\subsection{Stuff types}

It was remarked in the introduction that the combinatorial
representation of the categorified Heisenberg algebra extends the
Baez-Dolan groupoidification of the Fock space representation
\cite{finfeyn, mstuff}. In that situation, the spans $A$ and
$A^{\dag}$ act by weak pullback on \textit{stuff types}. A stuff type
consists of a groupoid $X$ `over $\FSN$'; that is, equipped with a
functor $\cat{X} \sxto{\Psi} \FSN$. A stuff type can be
interpreted as a generalized class of structures on finite sets. They
form the objects of the `over category' $\cat{Gpd}/\FSN$, whose
morphisms from $(X,\Psi)$ to $(X',\Psi')$ are maps $f : X \ra X'$
forming commuting triangles, so that $\Psi ' \circ f = \Psi$.

The over category $\cat{Gpd}/\FSN$ is taken in \cite{finfeyn, mstuff}
as a \textit{groupoidification} of Fock space, and the
degroupoidification of a stuff type gives a particular vector in Fock
space. This is sufficient to represent the 1-category
$\cat{Span_1(Gpd)}$, but not quite enough to represent the
2-morphisms. However, a closely related class of representations
follows immediately from the one we have described.

\begin{corollary}\label{cor:rep-on-hom}
  The representation on $\FSN$ in $\cat{Span(Gpd)}$ induces
  representations by functors and natural transformations on the
  categories $\Hom_{\cat{Span(Gpd)}}(\cat{G},\FSN)$.
\end{corollary}
\begin{proof}
  In the 2-category $\cat{Span(Gpd)}$, there are composition maps
  \begin{equation}
    \circ : \Hom(\cat{G},\FSN) \times \Hom(\FSN,\FSN) \ra \Hom(\cat{G},\FSN)
  \end{equation}
  This amounts to a functor:
  \begin{equation}
    \circ(-) : \Hom(\FSN,\FSN) \ra \Hom ( \Hom(\cat{G},\FSN) , \Hom(\cat{G},\FSN) ) 
  \end{equation}
  Note that the outer $\Hom$ on the right hand side is in $\cat{Cat}$,
  and the inner $\Hom$ operations are in $\cat{Span(Gpd)}$. which takes
  any 1-morphism from $\FSN$ to itself to the operation which acts by
  composition with that 1-morphism. Similarly, $\circ(-)$ takes a
  2-morphism in $\Hom(\FSN,\FSN)$ to a natural transformation between
  these functors. Composing this with the representation of Theorem \ref{thm:mainthm} gives the result.
\end{proof}

The first part of this argument would work in any 2-category, and does
not depend on any details about $\cat{Span(Gpd)}$. As a special case,
we have an action on a category whose objects correspond directly to
stuff types.

\begin{corollary}
  There is a representation of $\cat{H'}$ by functors and natural
  transformation on the category 
  \begin{equation}
    \Hom_{\cat{Span(Gpd)}}(\cat{1},\FSN) \cong \cat{Span(Gpd/\FSN)}
  \end{equation}
\end{corollary}
\begin{proof}
  The existence of such a representation on $\Hom_{\cat{Span(Gpd)}}(\cat{1},\FSN)$ is just a special case of Corollary \ref{cor:rep-on-hom}.

  The objects of this category are in 1-1 correspondence with stuff
  types, since any groupoid $X$ is equipped with a unique functor into
  $\cat{1}$, so a span from $\cat{1}$ to $\FSN$ determined uniquely by
  a stuff type. However, the morphisms between these spans are not
  just the maps of stuff types, but rather isomorphism classes of
  spans of such maps. That is, the morphisms of this category
  correspond exactly to those of $\cat{Span(Gpd/\FSN)}$, in a way
  which agrees with units and composition.
\end{proof}

\subsection{Linearization of Spans}

We will now consider a different way to get representations of
$\cat{K}$ on a category by functors and natural transformations,
which is also related to the Baez-Dolan groupoidification program, but
in a different way.

Linearization is the process of turning geometrical structures into
linear ones. An important form of linearization for our purposes is
\emph{degroupoidification}, first described in~\cite{finfeyn} and
surveyed in~\cite{hdavii}, which is a monoidal functor
\[
\cat{Span_1(Gpd)} \sxto D \cat{Vect},
\]
where \cat{Span_1(Gpd)} is the monoidal 1\-category obtained by
identifying isomorphic 1\-cells in the monoidal 2\-category
\cat{Span(Gpd)}. This functor $D$ takes groupoids to the free vector
space on their set of equivalence classes of objects, and spans of
groupoids to linear maps between these. Tameness properties of the spans are crucial for this functor $D$\ to be well-defined.

We will use a higher-categorical generalization of this called
\emph{2\-linearization}, which is a monoidal 2\-functor
\begin{equation}
\label{eq:lambda}
\cat{Span(Gpd)} \sxto \Lambda \cat{2Vect},
\end{equation}
where \cat{2Vect} is the monoidal 2\-category of
\textit{Kapranov-Voevodsky 2\_vector spaces}. This has objects given
by $\mathbb{C}$\-linear semisimple additive categories, 1\-morphisms
given by linear functors, and 2\_morphisms given by natural
transformations. Again, tameness properties of $\SiiG$ as described in Section~\ref{sec:forming2cat} are necessary for this construction to be valid.

The 2-functor $\Lambda$ is described in more detail
in~\cite{gpd2vect}, but the essential starting point is that $\Lambda$
takes a groupoid $\cat A$ to its category $\Rep(\cat A)$ of
finite-dimensional representations. A span of groupoids \mbox{$\cat B
  \sxlto G \cat X \sxto F \cat A$} is mapped to the composite functor
\begin{equation}
  \Rep(\cat A) \sxto{F_*} \Rep(\cat X) \sxto{G ^*} \Rep(\cat B),
\end{equation} 
where $F_*$ is the pullback functor (also known as the restriction
functor) for $F$, and $G^*$ is the adjoint of the pullback functor
(also known as the induction functor) for~$G$.

For a general span-of-spans of the form
\[
(\cat B \sxlto G \cat X \sxto F \cat A) \sxto {(\cat X \sxlto{S} \cat Z \sxto T \cat Y, \mu,\nu)} (\cat B \sxlto J \cat Y \sxto K \cat A),
\]
we construct its image under $\Lambda$ as the following composite natural transformation:
\begin{equation}
\begin{aligned}
\begin{tikzpicture}[xscale=4, yscale=2]
\node (X) at (0,0) {$\Rep(\cat X)$};
\node (Y) at (0,-2) {$\Rep (\cat Y)$};
\node (Z) at (0,-1) {$\Rep (\cat Z)$};
\node (B) at (-1,-1) {$\Rep (\cat B)$};
\node (A) at (1,-1) {$\Rep(\cat A)$};
\draw [->] (X) to node [auto, swap] {$G^*$} (B);
\draw [<-] (X) to node [auto] {$F_*$} (A);
\draw [<-] (Y) to node [auto, swap] {$K_*$} (A);
\draw [->] (Y) to node [auto] {$J^*$} (B);
\draw [<-] (Z) to [out=70, in=-70]
        node [right, inner sep=0pt]
        {$S_*$} (X);
\draw [->] (Z) to [out=110, in=-110]
        node [left, inner sep=0pt]
        {$S^* _{\vphantom{*}}$} (X);
\draw [<-] (Z) to [out=-70, in=70] node [right, inner sep=0pt] {$T_*$} (Y);
\draw [->] (Z) to [out=-110, in=110] node [left, inner sep=0pt] {$T^* {{}_{\vphantom{*}}}$} (Y);
\draw [->, double] (0.4,-0.7) to
    node [right, pos=0.5] {$\nu _*$}
    (0.4,-1.3);
\draw [->, double] (-0.6,-0.7) to
    node [right, pos=0.5] {$(\mu^*) ^{-1}$}
    (-0.6,-1.3);
\draw [->, double] ([yshift=-0.3cm] X.center) to
    ([yshift=0.3cm] Z.center);
\draw [->, double] ([yshift=-0.3cm] Z.center) to
    ([yshift=0.3cm] Y.center);
\end{tikzpicture}
\end{aligned}
\end{equation}
The central 2\-cells arise from adjunctions $S_* \dashv S^*$ and $T^*
\dashv T_*$. Notice that this makes essential use of the 2-sidedness
of the adjunctions between the restriction and induction functors
denoted with upper and lower stars respectively.

In Section~\ref{sec:khovanov} we described Khovanov's construction of
the categorified Heisenberg algebra as a monoidal category
\cat{H'}. In Section~\ref{sec:groupoidification}, we described how
this acts in a natural way by spans and spans-of-spans over the
groupoid $\FSN$ of finite sets and bijections, yielding a
combinatorial interpretation that comprises the definition of a 2-functor
\begin{equation}
\cat{\Omega_{H'}} \sxto C \cat{Span(Gpd)},
\end{equation}
We can compose $C$ and $\Lambda$ to
obtain a new representation of \cat{\Omega_{H'}} in \cat{2Vect}. In fact, this composite representation $\Lambda \circ C$ is exactly Khovanov's original representation as described in~\cite{khovanov}, up to completeness issues which we below.
\begin{corollary}
\label{cor:composite}
The composite 2-functor
\begin{equation}
\label{eq:composite2functor}
\cat{H'} \sxto C \cat{Span(Gpd)} \sxto \Lambda \cat{2Vect}
\end{equation}
determines a representation of \cat K on $\FV(\FSN)$ by functors and natural
transformations.
\end{corollary}

\noindent
This Corollary~\ref{cor:composite} would be immediate, except that $\Lambda$ was defined
only for essentially finite groupoids, whereas $\FSN$ is a \textit{union} of
essentially finite groupoids. In particular, it is a colimit of a
diagram of groupoids $(\cat{FS}_n)$, whose $n^{\mathrm{th}}$ member is the
groupoid of all finite sets with at most $n$ elements, with all the
natural inclusions. Thus Corollary~\ref{cor:composite} requires that $\Lambda$
be extended to groupoids of this kind (technically, Ind-objects for
finite groupoids) for the statement to be well-defined.

The extension of $\Lambda$, applied to $\FSN$, should thus give an
object of a category whose objects are again a colimit of finite
dimensional 2-vector spaces (Ind-objects for $\cat{2Vect}$). This is
exactly analogous to the case of an infinite direct sum of vector
spaces. So it consists of all finite-dimensional representations of
$\FSN$ (that is, finite direct sums of any of the irreducible
representations of any of the $\Sn$).

We will gloss over much discussion of this issue, since any such
object will eventually appear in the image $\Lambda(\cat{FS}_n)$ for
sufficiently large $n$, just as any object of $\FSN$ eventually
appears in sufficiently large $\cat{FS}_n$. Thus, any calculation we
might want to make with the extended $\Lambda$ will actually be a
calculation in some finite stage $\cat{FS}_n$.

\subsection{2-Linearized ladder operators}

We now investigate the images of the spans $A$ and
$A^{\dagger}$. This will make it clear why the
composite~\eqref{eq:composite2functor} is isomorphic to Khovanov's
representation.

To begin with, the 2-vector space analog of `Fock space' is:
\begin{equation}
  \FV(\FSN) = \HV{\FSN} \simeq \coprod_{n \in \mathbb{N}} \mathrm{Rep}(\Sn)
\end{equation}
This is a category generated by the irreducible representations of all
the symmetric groups $\Sn$.  The categorified Heisenberg algebra acts
on $\FV(\FSN)$ by endofunctors and natural transformations, and we
would like to understand this action concretely.

For each isomorphism class of object $n \in \FSN$, the irreducible representations of the group \mbox{$\mathrm{Aut}(n) = \Sn$} are indexed by Young diagrams with $n$ boxes, so there is a countable basis for $\FV(\FSN)$ consisting of all Young diagrams. See~\cite{youngtab} for the mechanics of Young diagrams in representation theory the symmetric groups, and~\cite{sternberg} for a good discussion in a physical context.

The 1\-cell $\Lambda(A ^\dagger)$ is given by the functor
\begin{equation}
\Lambda(\FSN) \sxto{\Lambda(+1)} \Lambda(\FSN),
\end{equation}
and $\Lambda(A)$ by its adjoint. For any $n$-element set, the functor
$(+1)$ maps the automorphism group $\Sn$ into $\Sni$, by the natural
inclusion where $\Sn$ is the subgroup of $\Sni$ fixing the newly-added
element.

To understand $\Lambda(A ^\dagger)$ we must therefore see how a
representation of $\FSN$ is pulled back by the functor $(+1)$, and to
understand $\Lambda(A)$ we must see how the adjoint of this pullback
operation acts. This is described by the theory of restricted and
induced representations, a standard part of the representation theory
of the symmetric groups~\cite{sagan}. Given an irreducible
representation of $\Sni$ represented by a Young diagram $D$, the
corresponding restricted representation of $\Sn$ will be a direct sum
of irreducibles, which correspond to all the valid Young diagrams
obtained by deleting a single box from some row of $D$.  Each
representation corresponding to such a diagram appears in the
restricted representation with multiplicity one. Likewise, given an
irreducible representation of $\Sn$ represented by Young diagram $D'$,
the induced representation of $\Sni$ is a representation in which
every Young diagram arising by adding a box to any row of $D'$
(possibly of zero length) appears with multiplicity one.

We can see the effect of $\FV(A)$ and $\FV(A^{\dagger})$ on small
representations of small symmetric groups using the following directed
graph of Young diagrams, which is known as Young's lattice:

\begin{equation}\label{eq:ydgraph}
{\tiny{\xy
(0,45)*+{\mathbb{C}}="0";
(0,30)*+{\yng(1)}="1";
(-10,15)*+{\yng(2)}="2";
(10,15)*+{\yng(1,1)}="11";
(-15,0)*+{\yng(3)}="3"; 
(0,0)*+{\yng(2,1)}="21"; 
(15,0)*+{\yng(1,1,1)}="111";
(-30,-15)*+{\yng(4)}="4"; 
(-15,-15)*+{\yng(3,1)}="31"; 
(0,-15)*+{\yng(2,2)}="22"; 
(15,-15)*+{\yng(2,1,1)}="211"; 
(30,-15)*+{\yng(1,1,1,1)}="1111";
(-45,-30)*+{\yng(5)}="5"; 
(-30,-30)*+{\yng(4,1)}="41"; 
(-15,-30)*+{\yng(3,2)}="32"; 
(0,-30)*+{\yng(3,1,1)}="311"; 
(15,-30)*+{\yng(2,2,1)}="221"; 
(30,-30)*+{\yng(2,1,1,1)}="2111"; 
(45,-30)*+{\yng(1,1,1,1,1)}="11111";
{\ar "0";"1"};
{\ar "1";"2"};
{\ar "1";"11"};
{\ar "2";"3"};
{\ar "2";"21"};
{\ar "11";"21"};
{\ar "11";"111"};
{\ar "3";"4"};
{\ar "3";"31"};
{\ar "21";"31"};
{\ar "21";"22"};
{\ar "21";"211"};
{\ar "111";"211"};
{\ar "111";"1111"};
{\ar "4";"5"};
{\ar "4";"41"};
{\ar "31";"41"};
{\ar "31";"32"};
{\ar "31";"311"};
{\ar "22";"32"};
{\ar "22";"221"};
{\ar "211";"221"};
{\ar "211";"311"};
{\ar "1111";"2111"};
{\ar "211";"2111"};
{\ar "1111";"11111"};
\endxy}}
\end{equation}
Given any diagram in this partial order, the effect of $\FV(A)$ is to
give the direct sum of every diagram immediately above it, and the
effect of $\FV(A^{\dagger})$ is to give the direct sum of every
diagram immediately below it.  As a result, the effect of multiple
applications of $\FV(A)$ and $\FV(A^{\dagger})$ can be calculated by
counting paths.

We can represent this graph in the following way:
\begin{equation}\label{eq:FVA}
\FV(A) = 
\begin{pmatrix}
0 & M_{0,1} & 0 & 0 & 0 & 0 & \dots \\
0 & 0 & M_{1,2} & 0 & 0 & 0 & \dots \\
0 & 0 & 0 & M_{2,3} & 0 & 0 & \dots \\
0 & 0 & 0 & 0 & M_{3,4} & 0 & \dots \\
0 & 0 & 0 & 0 & 0 & M_{4,5} & \dots \\
\vdots & \vdots & \vdots & \vdots & \vdots & \vdots & \ddots
\end{pmatrix}
\end{equation}
Entries in this matrix are themselves matrices, the first nontrivial
elements having the following forms:
\begin{equation}
M_{0,1} =
\begin{pmatrix}
\C
\end{pmatrix}
\end{equation}
\begin{equation}
M_{1,2} =
\begin{pmatrix}
\C & \C
\end{pmatrix}
\end{equation}
\begin{equation}
M_{2,3} =
\begin{pmatrix}
\C & \C & 0 \\
0 & \C & \C
\end{pmatrix}
\end{equation}
\begin{equation}
M_{3,4} =
\begin{pmatrix}
\C & \C & 0 & 0 & 0 \\
0 & \C & \C & \C & 0 \\
0 & 0 & 0 & \C & \C
\end{pmatrix}
\end{equation}
\begin{equation}
M_{4,5} =
\begin{pmatrix}
\C & \C & 0 & 0 & 0 & 0 & 0 \\
0 & \C & \C & \C & 0 & 0 & 0 \\
0 & 0 & \C & 0 & \C & 0 & 0 \\
0 & 0 & 0 & \C & \C & \C & 0 \\
0 & 0 & 0 & 0 & 0 & \C & \C \\
\end{pmatrix}
\end{equation}
We write these matrices using the basis order suggested by the graph
(\ref{eq:ydgraph}). The analogous matrices for $\FV(A^{\dagger})$ are
the transpose of these.

The commutation relation
\begin{equation}
\FV(A) \FV(A^{\dagger}) \oplus \mathbf{1} \cong \FV(A^{\dagger})\FV(A)
\end{equation}
must be true in \cat{2Vect}, since $\Lambda$ preserves all the
relevant structures, and since we have already demonstrated the
corresponding combinatorial fact~\eqref{eq:formalcommutationrelation}
about the composition of spans.  This fact can be directly verified
for actions on small symmetric groups using the matrices given above. In terms of diagrams, the extra factor of
$\mathbf{1}$ on the left-hand side corresponds to the operation of first adding a box after
the last row of a diagram $D$, then removing it.

This commutation relation recalls that the 2-linear maps given by the
$M_{i,i+1}$ and their adjoints are only the generators of the
categorified Heisenberg algebra. In physical applications, derived
operators are also very important. In particular there is the `number operator' $\mathbf{n}=\A^{\dagger}\A$. In the physical application of the
Heisenberg algebra to the quantum harmonic oscillator, this is the Hamiltonian of the system. It is a self-adjoint operator whose
eigenvalues correspond to the observable values of `energy', which is just the number of particles in a particular state of
the system.

With this physical motivation in mind, it is interesting to consider
the analogous span $N=A^{\dagger}A$, and its 2-linearization:
\begin{equation}
\FV(N) = \FV(A^{\dagger}) \FV(A) =
\begin{pmatrix}
N_0 & 0 & 0 & 0 & 0 & \dots \\
0 & N_1 & 0 & 0 & 0 & \dots \\
0 & 0 & N_2 & 0 & 0 & \dots \\
0 & 0 & 0 & N_3 & 0 & \dots \\
0 & 0 & 0 & 0 & N_4 & \dots \\
\vdots & \vdots & \vdots & \vdots & \vdots & \dots
\end{pmatrix}
\end{equation}
In place of the eigenvalues $n$ of the operator
$\mathbf{n}$, we have functors given by blocks such as the following:
\begin{equation}
N_4 =
\begin{pmatrix}
  \C & \C & 0 & 0 & 0 \\
  \C & \C^2 & \C & \C & 0 \\
  0 & \C & \C & \C & 0 \\
  0 & \C & \C & \C^2 & \C \\
  0 & 0 & 0 & \C & \C
\end{pmatrix}
\end{equation}
We observe here that the 2\-linear maps represented by these blocks
are endofunctors of particular subcategories of $\FV(\FSN)$. The
dimension of each entry counts paths of a certain shape
(lower-then-raise) in the Young lattice, but they are otherwise rather
opaque. However, the following observation shows that the $N_n$ are
quite natural, and indeed analogous to eigenvalues:

\begin{proposition}
There is a natural equivalence of 2\_linear maps $N_n \cong (-) \otimes \C[n]$, where $\C[n]$ is the regular representation of $\Sn$ on $\C^n$ acting by permutation of the standard basis.
\end{proposition}
\begin{proof}
  This can be verified by considering the combinatorial interpretation
  of the span $A^{\dagger}A$, which is removing an element from a set
  and then adding an element to the result set - this amounts to
  marking a chosen element of the original set, which can be done in
  exactly $n$ ways for an $n$-element set. These are all independent,
  and the symmetry group $\Sn$ acts on them exactly as the permutation
  representation.
\end{proof}

We close this section with a remark about how the above picture
captures Khovanov's model of the diagram calculus in the bimodule
category $\mathcal{S}'$. The blocks $M_{n,n+1}$ each act only on representations supported on
single objects $n$. They represent induction functors such as
\newcommand\Ind{\mathrm{Ind}}
\begin{equation}
  \Ind_{n}^{n+1} = (+1)|_{\mathbf{n}} : \Rep(\Sn) \ra \Rep(\Sni),
\end{equation}
which induces a representation from $\Sn$ to $\Sni$ along the
inclusion $\Sn \subset \Sni$.  The adjoint blocks describe the
restriction functors, showing how an $\Sni$ representation restricts
to an $\Sn$ representation.

Specifically, Khovanov defined a representation of $\cat{H'}$ in terms of
bimodules, as defined by a map from the category freely defined by the
graphical calculus given earlier, into the following:

\begin{definition}[Khovanov]Let $\mathcal{S}'$ be a category defined
  as follows. The objects of $\mathcal{S}'$ are composites of
  induction and restriction functors for representations of symmetric
  groups along the standard inclusions $\Sn \subset \Sni$ (and
  monoidal structure given by composition).  The morphisms are natural
  transformations.  Then $\mathcal{S} = Kar(\mathcal{S}')$ is the
  Karoubi envelope of $\mathcal{S}'$.
\end{definition}

\noindent
This is described in terms of groups, rather than the groupoid
$\FSN$, so it is constructed as the direct sum of a collection of such categories $\mathcal{S}_n$, labeled by the starting value of $n$. However, there is a faithful embedding
\begin{equation}
  \mathcal{S}' \ra \mathrm{End}_{\iiV}(\FV(\FSN))
\end{equation}
which takes a bimodule to the functor which acts by a taking the
tensor product (over the appropriate group algebra) with that
bimodule. For simplicity, we omit the distinction between
$\mathcal{S}'$ and the isomorphic copy of it lying in $\iiV$.

A general object in $\mathcal{S}'$ represents a functor which is a
composites of sequences of blocks like the $M_{n,n+1}$ above, or their
adjoints.  Taking the direct sum gives exactly the bimodules
corresponding to the functors in $\Hom \big(\Lambda(\FSN),\Lambda(\FSN)
\big)$ which are generated in this way.  Thus, the image in $\iiV$ of
our groupoidified Heisenberg algebra corresponds exactly to this
representation $\mathcal{S'}$ of $\cat{H'}$.

\subsection{Symmetrizers and antisymmetrizers}\label{sec:symmetrizers}

The algebra categorified by Khovanov is larger than the
single-variable Heisenberg algebra described in the introduction.  It
is a more complicated but closely related algebra also known as the
Heisenberg vertex algebra, isomorphic as an ordinary algebra to the
infinitely-generated Heisenberg algebra. It has a countably infinite
family of generators $a_i$ (which commute amongst themselves) and
their adjoints (which also commute amongst themselves), which obey
the following relations:
\begin{equation}
\A^{\pdag }_i \A_j^{\dagger} - \A_j^{\dagger} \A_i ^\pdag = \A_{j-1}^{\dagger} \A_{i-1} ^\pdag
\end{equation}
This includes the original relation when $i = j = 1$, if we take $\A_0
^\pdag = \A_0^{\dagger}=1$.  This larger algebra has a physical
interpretation in terms of conformal field theory, in which operators
are replaced by holomorphic operator-valued fields on a Riemann
surface.  This will not concern us here, though see lecture notes by
Thomas~\cite{thomas-HVA} for a starting point.

The categorification of this larger algebra uses a larger monoidal
category $\cat{H} = \mathrm{Kar}(\cat{H}')$.  This makes use of the
\textit{Karoubi envelope} of a category $\cat{C}$ (also known in
some contexts as the \textit{idempotent completion}, or
\textit{pseudoabelian hull}). The Karoubi envelope $\mathrm{Kar}(\cat{C})$
is a universal category containing $\cat{C}$ such that every
idempotent morphism splits. Concretely, up to isomorphism,
$\Kar(\cat{C})$ can be constructed as the category with objects
\begin{equation}
  \{ (c,p) | c \in \cat{C}, p : c \ra c \text{, $p$ idempotent}\},
\end{equation}
and with morphisms $(c_1,p_2)$ to $(c_2,p_2)$ given by $f : c_1 \ra c_2$
such that $p_2 \circ f = f = f \circ
p_1$.  This is interpreted as adding new objects so that each
idempotent has a kernel and cokernel which are subobjects of its
source and target respectively.  The idempotent $p$ can then be
interpreted as projection onto the cokernel.

In particular, this introduces symmetrizer and
antisymmetrizer idempotents associated to the action of $\Sn$ on
products such as $A^n$ by the maps $\hat{\mu}$.  The new objects
introduced can be described as symmetric or antisymmetric tensor
products of the object $Q_{\pm}$ (or $A$ and $A^{\dagger}$ in our
setting).  These are called $S^n_{\pm}$ and $\bigwedge^n_{\pm}$
respectively, and defined as $S_{\pm}^n = S^n(Q_{\pm})$ and
$\Lambda_{\pm} = \Lambda^n(Q_{\pm})$. It is a consequence of their
construction as symmetrizers and antisymmetrizers that they satisfy
relations which are the analogs of the basic commutation relations for
the single-variable Heisenberg algebra:
\begin{equation}\label{eq:khoviso}
S_s^n \otimes \Lambda_+^m \cong (\Lambda_+^m \otimes S_-^n) \oplus (\Lambda_+^{m-1} \otimes S_-^{n-1})
\end{equation}
See Khovanov's paper~\cite{khovanov} for more details; there is nothing
substantially different in our setting.

These summands do not exist as spans of groupoids, where the direct
sum is just the disjoint union of spans, and as a result this aspect
cannot be groupoidified in our combinatorial model.  They do appear,
however, after we apply the 2-functor $\FV$, since $\iiV$ already contains
all such subobjects. This is a reflection of two basic facts: first,
that 2-vector spaces are abelian categories; second, that $\iiV$ is
compact closed (a version of this fact is proved in Yetter
\cite{yetter-cla}, Theorem 28.)

\begin{lemma}\label{lemma:karoubi-iiV}
  Given a monoidal category $\cat{C}$, any inclusion $i$ of
  $\cat{C}$ into the category of endofunctors of an object $\cat{V} \in
  \iiV$, extends to an inclusion $i'$ of $\Kar(\cat C)$:
\begin{equation}
\label{eq:karoubi-iiV-commute}
\begin{aligned}
\begin{tikzpicture}[yscale=1.5, xscale=2.5]
\node (1) at (0,0) {$\cat {C}$};
\node (2) at (1,0) {$\mathrm{End}_\cat{2Vect} (X)$};
\node (3) at (0,-1) {$\Kar (\cat C)$};
\draw [->] (1) to node [auto] {$i$} (2);
\draw [->] (1) to node [auto, swap] {} (3);
\draw [->] (3) to node [auto, swap] {$i'$} (2);
\end{tikzpicture}
\end{aligned}
\end{equation}
\end{lemma}
\begin{proof}
  The functor category $\mathrm{End}_{\iiV}(\cat{V}) \simeq \cat{V}^{op}
  \otimes \cat{V}$ is also a 2-vector space, hence in particular is an
  abelian category. Thus, every morphism has a cokernel. Then we
  define $i'$ on objects by
\begin{align}
    i'(c,p) &= \mathrm{coker}(i(p)).
\end{align}
On morphisms, $i'$ is given by
  defining
  \begin{equation}
    i' \bigl{(} f : (c_1,p_1) \ra (c_2,p_2) \bigr{)}
  \end{equation}
to be the corresponding map between the subspaces $\mathrm{coker}(p_1)$ and $\mathrm{coker}(p_2)$.  This $i'$ is well-defined since by hypothesis $f$ commutes with the projections to the cokernels. It is a functor because it respects identities and composition.
\end{proof}

Schematically we have the following diagram, treating monoidal
categories as one-object 2-categories:
\begin{equation}
\begin{aligned}
\begin{tikzpicture}[yscale=1.5, xscale=2.5]
\node (1) at (0,0) {$\cat {\Omega_{H'}}$};
\node (2) at (1,0) {$\cat{Span(Gpd)}$};
\node (3) at (0,-1) {$\cat{\Omega_H}$};
\node (4) at (1,-1) {$\cat{2Vect}$};
\draw [->] (1) to node [auto] {$C$} (2);
\draw [->] (1) to node [auto, swap, inner sep=1pt] {$F'$} (4);
\draw [->] (3) to node [auto, swap] {$F$} (4);
\draw [->] (2) to node [auto] {$\Lambda$} (4);
\draw [->] (1) to node [auto, swap] {} (3);
\end{tikzpicture}
\end{aligned}
\end{equation}
This says that the representation of $\cat{H}'$ into $\iiV$, which factors through $\SiiG$, extends in this compatible way to a representation of $\cat{H}$.

\subsection{Symmetrizers in $\iiV$}

We will consider here how the symmetrizers appear in concrete terms
within $\iiV$ after 2-linearizing our combinatorial interpretation.

The natural isomorphism
\begin{equation}
\label{eq:actionsigmahat}
\Lambda(A) \circ \Lambda(A) \sxto {\Lambda(\hat{\sigma})} \Lambda(A) \circ \Lambda(A) 
\end{equation}
that permutes the order of the operators $\Lambda(A)$ can be
represented in matrix form as a collection of linear maps, each acting
on a component of the matrix representation of $\FV(A) \circ \FV(A)$.
That is, each of the vector spaces that appear in this matrix carries
a particular action of the symmetric group $S_2$.  The symmetrizer
operation selects the subspace which is fixed by this action, and the
antisymmetrizer operation selects the complement of this subspace.

More generally, each component in $A^n$ is a representation of $\Sn$,
which decomposes into irreducibles, and the symmetrizer takes the
sub-representation consisting of exactly the copies of the trivial
representation which appear in this decomposition.  Similarly, the
antisymmetrizer takes only the copies of the sign representation.
For any irreducible representation of $\Sn$ there will be a
correspondingly `symmetrized' sub-functor of $A^n$. These operations
are the `Young symmetrizers' associated to the Young diagrams which
index such representations.

In particular, the subobjects which are needed to produce all the
generators of the Heisenberg algebra are those arising from the
symmetrizer idempotent for $\FV(A^n)$ or $\FV((A^{\dagger})^n)$ which
come from the symmetric group action
\begin{align}
\label{eq:symmetrizer}
  v &\mapsto \sum_{\mu \in \Sn} \hat{\mu}(v)
\intertext{and the antisymmetrizer}
\label{eq:antisymmetrizer}
  v &\mapsto \sum_{\mu \in \Sn} \mathrm{sign}(\mu) \hat{\mu}(v) \, .
\end{align}
We can illustrate this by considering the product $\FV(A \circ A)
\cong \FV(A) \circ \FV(A)$.  The nontrivial permutation $\sigma$ of
the 2-element set gives the 2-cell
\begin{equation}
A \circ A \sxto {\hat{\sigma}} A \circ A,
\end{equation}
and thus, combined with the natural isomorphism above, we obtain the
natural isomorphism~\eqref{eq:actionsigmahat}.

The composite $\FV(A) \circ \FV(A)$ is zero except for blocks of the form
\begin{equation}
M ^2_{i,i+2} = M ^{} _{i,i+1} \otimes M ^{} _{i+1,i+2} \,.
\end{equation}
From here on, we will use the notation $M_{i,j}$ to denote the matrix
of vector space which is the $(i,j)\mathrm{th}$ block entry of the
matrix $\FV(A)^{j-i}$. Up to a natural isomorphism, these can be
expressed in matrix form:
\begin{equation}
M_{0,2} =
\begin{pmatrix}
\C & \C
\end{pmatrix}
\end{equation}
\begin{equation}
M_{1,3} =
\begin{pmatrix}
\C & \C^2 & \C
\end{pmatrix}
\end{equation}
\begin{equation}
M_{2,4} =
\begin{pmatrix}
\C & \C^2 & \C & \C & 0 \\
0 & \C & \C & \C^2 & \C
\end{pmatrix}
\end{equation}
As with the basic blocks $M_{n,n+1} $, a general block $M _{i,j}$ can
be interpreted as giving a decomposition in matrix form of the
$\C[S_i] {-} \C[S_j]$--bimodule which effects the induction functor
\begin{equation}
\Rep(S_i) \sxto{\raisebox{3pt}{$\Ind _{S_i} ^{S_j}$}} \Rep(S_j)
\end{equation}
associated to the natural inclusion of $S_i$ into $S_j$.  That is,
given a representation \mbox{$\rho \in \Rep(S_i)$}, the induced
representation of $S_j$ is $\rho \otimes_{\C[S_i]} M_{i,j}$.

\begin{example}
  The block $M_{2,4}$ has component $\C^2$ in the position indexed by
  $\left ( \!\inalign{\tiny\yng(2)} \, , \inalign{\tiny\yng(3,1)}
    \,\right)$.  As we have remarked, this reflects the fact that
  there are two paths in the directed graph (\ref{eq:ydgraph}) from
  {\tiny{$\yng(2)$}} to {\tiny{$\inalign{\yng(3,1)}$}}, passing
  through {\tiny{$\yng(3)$}} or {\tiny{$\inalign{\yng(2,1)}$}}.

  A Young tableau is a way of filling the boxes of the diagram with
  the numbers $1, \dots, n$ such that they denote a valid order in
  which to add the boxes while following a path through
  (\ref{eq:ydgraph}). It is a standard fact that the tableaux which
  can fill a given Young diagram label basis elements of the
  irreducible representation labeled by that diagram.  Thus, we see
  that the component $\C^2$ acts to exchange the two basis elements.
  This is a permutation representation on the basis of $\C^2$, which
  decomposes as:
  \begin{equation}
    \inalign{\yng(2)} \,\oplus\, \inalign{\yng(1,1)}
  \end{equation}
  The symmetrizer and antisymmetrizer each select just one of these
  summands.  The two other components are both trivial $S_2$
  representations.  So the symmetrized functor has a matrix form with
  a $(2,4)$ block:
  \begin{equation}
    \begin{pmatrix}
      \C & \C & \C & \C & 0 \\
      0 & \C & \C & \C & \C
    \end{pmatrix}
  \end{equation}
  The antisymmetrized functor has the following block:
  \begin{equation}
    \begin{pmatrix}
      0 & \C & 0 & 0 & 0 \\
      0 & 0 & 0 & \C & 0
    \end{pmatrix}
  \end{equation}
\end{example}

\begin{example}
  The situation becomes more complex for $A^n$ with larger values of
  $n$. We now look at the first case where there are three new boxes
  to add, and the order is unconstrained. Consider the product
  $\FV(A^3) \cong \FV(A)^3$, and in particular look at the block
  \begin{equation}
    M_{3,6} = 
    \begin{pmatrix}
      \C & \C^3 & \C^3 & \C^3 & \C & \C^2 & 0 & \C & 0 & 0 & 0 \\
      0 & \C & \C^3 & \C^3 & \C^2 & \C^6 & \C^2 & \C^3 & \C^3 & \C & 0 \\
      0 & 0 & 0 & \C & 0 & \C^2 & \C & \C^3 & \C^3 & \C^3 & \C
    \end{pmatrix}
  \end{equation}
  This can be easily checked by calculation, but in particular, the
  component indexed by {\tiny{$\inalign{\yng(2,1)}$}} and
  {\tiny{$\inalign{\yng(3,2,1)}$}} is $\C^6$, which reflects the fact
  that there are six paths through Young's lattice from the former to
  the latter. In particular, however, these paths form a cube, since
  the boxes may be added in any order.
  
  By an analogous argument to the one in the previous example, this
  becomes a representation of $S_3$ by permutations of the three new
  boxes, and hence the axes of the cube of paths in Young's lattice.
  The basis for the component $\C^6$ consists of the labellings of the
  three boxes by numbers $\{ 4,5,6 \}$.  These can be appended to any
  Young tableau in {\tiny{$\inalign{\yng(2,1)}$}}, and thereby
  generate a basis for the tensor product with the bimodule which $M
  ^3_{3,6}$ depicts.  Thus, $S_3$ acts by permutations on the set of
  orders of $\{ 4,5,6 \}$.  This is equivalent, as an $S_3$
  representation, to the action on the regular representation
  $\C[S_3]$ itself.  This decomposes as
  \begin{equation}
    \inalign{\yng(3)} \,\oplus\, 2 \inalign{\yng(2,1)} \,\oplus\, \inalign{\yng(1,1,1)}
  \end{equation}
  So again, we have a one-dimensional subspace in this position for
  either the symmetrized or antisymmetrized product. This block is the
  first case where the other Young symmetrizers, in this case the one
  associated to {\tiny{$\inalign{\yng(2,1)}$}}, will give a nontrivial
  component.
\end{example}

We have seen here that the functors $\FV(A)$ and $\FV(A^{\dagger})$
have, as sub-functors, precisely the symmetric and antisymmetric
powers $S^n_{\pm}$ and $\bigwedge^n_{\pm}$ which will give the actual
generators of the multivariable Heisenberg algebra when passing to the
Grothendieck ring. So, after taking the image under $\FV$ of our
groupoidification, passing to the completion of the image recovers
Khovanov's full categorification. This is closely related to a
construction used by Khovanov in the proof that there is a map
$H_{\mathbb{Z}}$ from the integral form of the Heisenberg algebra into
the Grothendieck group of $\cat{H}$.

\section{Combinatorial models of categorified $\Usln$}\label{sec:Usln}
\label{sec:sln}

\subsection{Introduction}

We now investigate how the techniques described so far can be applied
to find combinatorial models of categorifications of $\Usln$, the
universal enveloping algebras of the Lie algebras $\sln$.

These categorifications are in the style of the Khovanov-Lauda
programme~\cite{KL, lauda-sl2-intro, lauda-sl2}. The main objects of
interest in that programme are categorifications of the $q$-deformed
enveloping algebras, the quantum groups $\Uqsln$. However, the
combinatorial interpretation we describe here only applies to the case
$q=1$, in which case the Grothendieck ring is a module over the
integers $\mathbb{Z}$.  In the $q$-deformed situation, one might
expect to obtain a module over the ring $\mathbb{Z}[q,q^{-1}]$ of
formal power series in $q$.

We expect that a similar combinatorial model should exist of the
categorification of the full $q$-deformed algebra, following the
pattern of the groupoidification of Hecke algebras~\cite{hdavii}.  In
this approach, the groupoid of finite sets is replaced with the
groupoid of finite-dimensional vector spaces over a finite fields
$\mathbb{F}_q$.  The special linear groups for such vector spaces take
the role of symmetric groups $\Sn$, and their size is related to
$q$-factorials.

\subsection{The algebras $\Usln$}

We write $\sln$ for the Lie algebra associated to the special linear
group $SL(n)$, which is given as the group of $n \times n$ matrices
with unit determinant.  Its Lie algebra would then consist of the
traceless $n \times n$ matrices.  It is presented as an abstract
algebra in terms of generators $h_i$, $e_i$ and $f_i$ for $i \in \{1,
\dots ,n\!-\!1\}$, satisfying the following relations:
\begin{align}
\label{eq:slnrelation1}
[e_i, f_j] &= \delta _{ij} h_i
\\
[e_i, h_j] &= 2 \delta _{ij} e_i
\\
\label{eq:slnrelation3}
[f_i, h_j] &= 2\delta_{ij} f_i
\end{align}
This is a special case of the standard presentation, in the
Chevalley basis, of the Lie algebra generated by Cartan data associated
to a particular Dynkin diagram~\cite{gilmore, fuchschw}.  In the
representation of $\mathrm{sl}_2$ as the traceless matrices acting on
$\mathbb{C}^2$, these elements have the following concrete
realization:
\begin{equation}
e = \begin{pmatrix} 0 & 1 \\ 0 & 0\end{pmatrix}
\qquad
f = \begin{pmatrix} 0 & 0 \\ 1 & 0 \end{pmatrix}
\qquad
h = \begin{pmatrix} 1 & 0 \\ 0 & -1 \end{pmatrix}
\end{equation}
Just as the Heisenberg algebra has a natural representation on the
polynomial ring $\Cz$, the universal enveloping algebras $\Usln$ have
natural representations on polynomial rings $\Czz$, defined in the
following way:
\begin{equation}
\Czz = \bigoplus_j \big(\mathbb{C}^n \big)^{\otimes_s j}
\end{equation}
In this polynomial representation, the generators $e_i$ and $f_i$ of $\Usln$ become
\begin{align}
\label{eq:slnrep1}
e_i &= z_{i+1} \partial_i \\
f_i &= z_i \partial_{i+1},
\intertext{and their commutator is}
\label{eq:slnrep3}
h_i &= z_{i+1} \partial_{i+1} - z_i \partial_i.
\end{align}
This extension gives back the defining representation of $\sln$ when
$\mathbb{C}^n$ is identified with the space of degree-1 polynomials.

This representation of $\Usln$ suggests defining
\begin{align}
n_i &:= z_i \partial_i,
\intertext{so that we have}
h_i &= n_{i+1} - n_i.
\end{align}
Rewriting our Lie algebra
relations~(\ref{eq:slnrep1}\_\ref{eq:slnrep3}) in terms of these new
variables, and expanding out commutators explicitly and reorganizing
to eliminate negatives, we obtain the following:
\begin{align}
\label{eq:catrel1}
e_i f_j + \delta _{ij} n_i &= f_j e_i + \delta _{ij} n_{i+1}
\\
e_i n_{j+1} + n_j e_i &= e_i n_j + n_{j+1} e_i + 2 \delta_{ij} e_i
\\
\label{eq:catrel3}
f_i n_{j+1} + n_j f_i &= f_i n_j + n_{j+1} f_i + 2 \delta_{ij} f_i
\end{align}
In this form, the relations are now suitable for groupoidification.
We note that these relations are a special case of the general pattern
for the construction of a Kac-Moody algebra from a Cartan matrix.  In
the case of these Lie algebras, the Cartan matrix is determined by the
Dynkin diagram associated to the algebra.  So the pattern for
groupoidifying these algebras should apply in these more general cases
as well.

\subsection{Groupoidification}

The groupoidification of $\sln$ is based on the groupoid of
$n$-coloured finite sets and colour-preserving bijections, written
$\nCFSN$. This is equivalent to the category where objects are
$n$-tuples of finite sets and morphisms are $n$-tuples of bijections,
which is the Cartesian product of $n$ copies of \FSN. The sets of
colour $i$ form the $i$th component of the cartesian product. A third
way to describe $\FSN ^n$ up to equivalence is as the free symmetric
monoidal groupoid on the discrete groupoid with $n$ objects.

\newcommand\p[1]{+1_#1} We write $\FSN^n \sxto {\p{i}} \FSN^n$ for the
functor which acts as $+1$ on the sets of colour $i$, and the identity
otherwise. For any $i \in \{ 1 , \ldots , n \}$, we use this to define
the spans $A _i ^\pdag$ and $A_i ^\dag$ as follows:
\begin{equation}
\begin{aligned}
\begin{tikzpicture}[xscale=2, yscale=1.5]
\node (T) at (0,0) {$\FSN^n$};
\node (L) at (-1,-1) {$\FSN^n$};
\node (R) at (1,-1) {$\FSN^n$};
\draw [->] (T) to node [auto, swap] {\id} (L);
\draw [->] (T) to node [auto] {$\p i$} (R);
\end{tikzpicture}
\end{aligned}
\hspace{70pt}
\begin{aligned}
\begin{tikzpicture}[xscale=2, yscale=1.5]
\node (T) at (0,0) {$\FSN^n$};
\node (L) at (-1,-1) {$\FSN^n$};
\node (R) at (1,-1) {$\FSN^n$};
\draw [->] (T) to node [auto, swap] {$\p i$} (L);
\draw [->] (T) to node [auto] {\id} (R);
\end{tikzpicture}
\end{aligned}
\end{equation}
These are `multicoloured' variants of the ordinary annihilation and
creation spans defined in equation~\eqref{eq:loweringraising}. They
satisfy an adjusted version of the categorified commutation relation:
\begin{equation}
A_i ^\pdag \circ A _j ^\dag \simeq A _j ^\dag \circ A_i ^\pdag \oplus \delta _{ij} \id_{\FSN ^n}
\end{equation}
Here $\delta _{ij}$ represents the identity span if $i=j$, and the zero span otherwise.

To obtain our model of categorified $\Usln$, we arbitrarily assign a
total ordering to our set $n$ of colours. We then make the following
definitions for all $i \in \{1,\ldots,n\!-\!1\}$:
\begin{align}
E_i &:= A ^\dag _{i+1} \circ A _i ^\pdag
\\
F_i &:= A ^\dag _i \circ A _{i+1} ^\pdag
\intertext{These are mutually-converse spans for each $i$. We also define the number operator $N_i$ for each colour $i$ as follows:}
N_i &:= A_i ^\dag \circ A_i ^\pdag
\end{align}
This span is its own converse. The operators $E_i$, $F_i$ and $N_i$
satisfy a categorified form of the reorganized
relations~\mbox{(\ref{eq:catrel1}\_\ref{eq:catrel3})} for $\Usln$:
\begin{align}
\label{eq:firstcatsln}
E_i F_j \oplus \delta _{ij} N_i &\simeq F_j E_i \oplus \delta _{ij} N_{i+1}
\\
E_i N_{j+1} \oplus N_j E_i &\simeq E_i N_j \oplus N_{j+1} E_i \oplus 2 \delta_{ij} E_i
\\
F_i N_{j+1} \oplus N_j F_i &\simeq F_i N_j \oplus N_{j+1} F_i \oplus 2 \delta_{ij} F_i
\end{align}
The reason for this is combinatorial, and can be understood by
explicitly tracking histories. For instance, in the first relation,
consider the span $E_i F_j$.  This is equal to $A^{\dag}_{j} \circ
A_i^{\pdag} \circ A_i ^\dag \circ A_{j} ^{\pdag}$, a sequence of
adding and removing elements of colours $i$ and $j$.  In the case
where $i \neq j$ the operators $A_{i}$ and $A_{j}$ commute, so this is
isomorphic to $A_i^{\pdag} \circ A_i^{\dag} \circ A_{j}^{\dag} \circ
A_{j}^{\pdag}$.  Since the spans $A_i$ and $A_j$ each satisfy the
usual categorified commutation relations for the Heisenberg algebra,
this is isomorphic to $(N_i \oplus \id) \circ N_j \simeq N_i N_j
\oplus N_j$.  A similar computation holds on the other side, and so
the isomorphism~\eqref{eq:firstcatsln} can be established.  Applied to
a coloured set with $n_i$ elements of colour $i$, and $n_j$ elements
of colour $j$, both sides compute $n_i n_j + n_i + n_j$.  If we add
the identity on each side, this relation for $E_i$ and $F_j$ witnesses
the isomorphism of two different ways to count $(n_i +1)(n_j +1)$, the
number of ways to distinguish either zero or one elements of each
colour in any given set.

Applying the degroupoidification functor~\eqref{eq:lambda} to $E_i$,
$F_i$ and $N_i$ recovers the standard
action~(\ref{eq:slnrelation1}\_\ref{eq:slnrelation3}) of $\Usln$ on
spaces of polynomials.  The coefficients obtained by the
multiplication and differentiation are exactly the numbers we have
just calculated in the example.  Just as with the Heisenberg algebra,
the algebraic facts about the relations they satisfy now appear as
consequences of combinatorial facts about coloured sets. In this way,
we have a groupoidification of the concrete representation on $\Czz$
of the algebra $\Usln$.  In the same way, the abstract
categorifications of Khovanov-Lauda type are realized concretely by
endomorphisms of an object in $\SiiG$.

We can make a further comment about this groupoidification from a
physical point of view.  As described in
Section~\ref{sec:combinatorics}, the groupoidification of `states' in
Fock space for the quantum harmonic oscillator can be described as
stuff types, a generalization of combinatorial species.  That is, such
a state is a groupoid $\cat X$ equipped with a functor $\cat X \sxto
\Phi \FSN$.  We think of $\cat X$ as an ensemble of structures whose
underlying finite sets are specified by $\Phi$, and which have a
notion of symmetry that is preserved by $\Phi$.  Combinatorial species
are exactly the case where $\Phi$ is faithful, meaning that morphisms
of $\cat X$ are determined by the underlying bijection of finite
sets. Considering a stuff type as a span $\cat X \sxlto \Phi\cat S \to
\cat 1$ and degroupoidifying it, we obtain a linear map $\mathbb{C}
\to \Cz$, which specifies a power series in $z$ by considering the
image of $1 \in \mathbb C$. This is the generating function in the
ordinary sense for the combinatorial structure~\cite{wilf}.

In the same way, a state in the Fock space for our groupoidified
$\Usln$ can be described as a groupoid $\cat X$ equipped with $\cat X
\sxto \Psi \nCFSN$, thought of as a span out of $\cat{1}$.  This is a
generalization of a \textit{multisort species} \cite{bll} which
describe structures on coloured sets, in the same way that stuff types
generalize species.  Thus, we can describe such a state as a
\textit{multisort stuff type}.  Again, degroupoidification of such a
span determines an element of $\Czz$, so a multisort stuff type
determines a vector in the multiple-variable Fock space.  The
groupoidified $\Usln$ acts on mulitsort stuff types just as the
groupoidified Heisenberg algebra acts on stuff types.

\bibliographystyle{plain}
\bibliography{references}

\end{document}